\documentclass[reqno]{amsart}
\usepackage[all]{xy}
\usepackage{amsmath}
\usepackage{amsfonts}
\usepackage{amssymb}
\usepackage{amsthm}
\usepackage{amscd}
\usepackage{latexsym}
\usepackage{txfonts}
\usepackage{tikz}
\usepackage{graphicx}
\usepackage{float}
\usepackage{multirow}
\usepackage{float}
\usepackage{threeparttable}
\usepackage{enumitem}

\DeclareMathOperator{\Ima}{im}
\newcommand{\bb}{\mathbb}
\newcommand{\lp}{\ell}
\newcommand{\G}{G(V)}
\newcommand{\Po}{\bb{P}_{0}}
\newcommand{\ov}{\overline}
\theoremstyle{plain}
\newtheorem{theorem}{Theorem}
\newtheorem{lemma}[theorem]{Lemma}

\newtheorem{defi}[theorem]{Definition}
\newtheorem*{theorem*}{Theorem}
\theoremstyle{remark} \newtheorem*{remark}{Remark}
\theoremstyle{remark} \newtheorem{example}{Example}

\makeatletter
\def\paragraph{\@startsection{paragraph}{4}%
  \z@\z@{-\fontdimen2\font}%
  {\normalfont\bfseries}}
\makeatother

\begin{document}
\title{Smith and Critical groups of Polar graphs.}
\date{}
\author{Venkata Raghu Tej Pantangi}
\email{pvrt1990@ufl.edu}
\author{Peter Sin.}
\email{sin@ufl.edu}
\address{Department of Mathematics, University of Florida, Gainesville, FL 32611-8105, USA.}
\thanks{This work was partially supported by a grant from the Simons Foundation (\#204181 to Peter Sin).}
\begin{abstract}
We compute the elementary divisors of the adjacency and Laplacian matrices of families of polar graphs. These graphs have as vertices the isotropic one-dimensional subspaces of finite vector spaces with respect to non-degenerate forms, with adjacency given by orthogonality. 
\end{abstract}

\maketitle

\section{Introduction}
Let $\Gamma=(\tilde{V},\tilde{E})$ be an undirected simple connected graph. Let $A$ be the adjacency matrix of $\Gamma$ with respect to some arbitrary ordering of $\tilde{V}$. Let $D$ be the diagonal matrix with $D_{ii}$ being the degree of the $i$th vertex of $\Gamma$. Then $L:=D-A$ is called the Laplacian matrix of $\Gamma$. Multiplication by $A$ and $L$ are endomorphisms of $\bb{Z}{\tilde{V}}$, the free $\bb{Z}$ module with $\tilde{V}$ as a basis set. The cokernel of $A$ is called the \emph{Smith group} $S(\Gamma)$ of $\Gamma$. The finite part of cokernel of $L$ is called the \emph{critical group} $K(\Gamma)$ of $\Gamma$. As a consequence of Kirchhoff's Matrix Tree Theorem, the order of $K(\Gamma)$ is equal to the number of spanning trees of $\Gamma$ (cf. \cite{Sta}). The critical groups of various graphs arise in combinatorics in the context of chip firing games (cf. \cite{Biggs}), as the abelian sandpile group in statistical mechanics (cf. \cite{Dhar}), and also in arithmetic geometry. One may refer to \cite{Lor1} for a discussion on these connections. It is therefore of some interest to compute the Smith groups and critical groups of graphs.

In this paper, we calculate the Smith groups and critical groups of families of {\it polar graphs}. 
 These graphs are strongly regular graphs (SRGs) associated with finite classical polar spaces. Given a finite classical polar space $\mathcal{P}$ associated with a quadratic, symplectic or hermitian space $V$, the polar graph $\Gamma(V)$ is the graph whose vertex set is the set $\Po$ of points of  $\mathcal{P}$ (isotropic $1$-dimensional subspaces of $V$) and  whose adjacency is defined by orthogonality with respect to the underlying form. We have six infinite families of polar graphs arising from the six families of finite classical polar spaces.  
 Throughout the paper the Smith and critical groups of $\Gamma(V)$ will be denoted by $S$ and $K$ respectively. The computation of $S$ and $K$ is equivalent 
to finding the elementary divisors of $A$ and $L$, so can be carried out one prime at a time. The group $\G$ of form-preserving isomorphisms on $V$ acts a group of automorphisms of $\Gamma(V)$. We use properties of strongly regular graphs and modular representation theory of $\G$ to compute $S$ and $K$. This methodology is an extension of the methodologies used for other computations in the literature, as found in \cite{SDB}, \cite{CSX}, \cite{S} and \cite{SD}.  

Our results giving the elementary divisors of $S$ and $K$ for every polar graph
are stated in \S\ref{main_results}. The structures of $K$ and $S$ depend on the
relationship of primes to various combinatorial parameters of $\Gamma(V)$
which we treat first in \S\ref{nota}. Then in \S\ref{main_results} the 
multiplicity of elementary divisors are given in twelve tables, corresponding
to $S$ and $K$ for each of the six families of polar spaces.

We now describe our approach in more detail.
Let us consider a matrix $X \in M_{n \times n}(\bb{Z})$, a prime $\lp$, and a positive integer $a\in \bb{Z}_{\geq 0}$. By $e_{0}$ we denote the $\lp$-rank of $X$ and for $a\neq 0$ let $e_{a}$ be the multiplicity of $\lp^{a}$ as an elementary divisor of the finite part of $\bb{Z}^n/ X\left(\bb{Z}^n\right)$. Let $\bb{Z}_{\lp}$ be the ring of $\lp$-adic integers and $\bb{F}_{\lp}=\bb{Z}_{\lp}/ \lp \bb{Z}_{\lp}$ the field of $\ell$-elements. 
 In \S\ref{snf}, we establish some general results on elementary divisors by relating them to
the $\bb{Z}_{\lp}$-modules $M_{a}:=\{y \in \bb{Z}_{\lp}^{n}| Xy \in \lp^{a}\bb{Z}_{\lp}^{n}\}$ and the $\bb{F}_{\lp}$-modules $\ov{M}_{a}:= (M_{a}+\lp \bb{Z}_{\lp}^{n})/ \lp \bb{Z}_{\lp}^{n}$. By the structure theorem
for finitely generated modules over a principal ideal domain, we have $e_{a}= \dim(\ov{M}_{a}/ \ov{M}_{a+1})$. 

Suppose $X$ is now either the adjacency matrix $A$ or the Laplacian matrix $L$ of $\Gamma(V)$. As $\Gamma(V)$ is a strongly regular graph, $X$ has three distinct eigenvalues. When a prime $\lp$ divides at most two of the three eigenvalues,  the reduction $\ov{X}_{\lp}$ of $X$ modulo $\lp$ has non-zero eigenvalues. If $\lp$ divides all three eigenvalues, then $\ov{X}_{\lp}$ is nilpotent. The division into the nilpotent and non-nilpotent cases
is an important one from a technical point of view, and is also reflected in structure of $S$ and $K$. 
The nilpotence of $\ov{X}_{\lp}$ is characterized by a few arithmetic conditions together with the geometry on $V$. Table \ref{nilconst} of \S\ref{primes} encodes this characterization. 

The non-nilpotent case turns out to be much easier to handle
than the nilpotent case; general properties of strongly regular graphs and elementary linear algebra suffice.  In \S\ref{easy} we compute the $\lp$-elementary divisors of $S$ and $K$ when $\ov{X}_{\lp}$ is not nilpotent.  In this case, we  show that each $\ov{M_{i}}$ is either $\Ima(\ov{X}_{\lp}-\ov{\alpha} I)$ or $\ov{\ker(X-\alpha I)}$ for some eigenvalue $\alpha$ of $X$. 

When  $\ov{X}_{\lp}$ is nilpotent, we use more representation theory.
As $\G$ preserves adjacency, $X$ may be viewed as a $\bb{Z}_{\lp}\G$-endomorphism of the $\bb{Z}_{\lp}\G$-permutation module $\bb{Z}_{\lp}\Po$, and the modules $\ov{M}_{a}$ introduced above are $\bb{F}_{\lp}\G$-submodules of the $\bb{F}_{\lp}\G$-permutation module
$\bb{F}_{\lp}\Po$. 
The action of $\G$ on $\Po$ has permutation rank $3$. The submodule structure of the permutation module $\bb{F}_{\lp}\Po$ has been determined in \cite{Lie1}, \cite{Lie2}, \cite{LST}, and \cite{ST} in {\it cross-characteristics}, that is, when
$\lp$ is not equal to the characteristic of the underlying field of the polar space. 
From the submodule structures determined in \cite{Lie1}, \cite{Lie2}, \cite{LST}, and \cite{ST} it can be seen that the length of $\bb{F}_{\lp}\Po$ is at most $6$. Thus we expect at most five  distinct $\lp$-elementary divisors for $S$ and $K$.  Using the information
about submodules we locate the submodules $\ov{M}_{i}$ in the submodule lattice and thereby determine the dimensions of the submodules $\ov{M}_{i}$. This detailed analysis, which forms the bulk of the paper, is carried out in \S\ref{act} to \S\ref{last}. The identification of the submodules $\ov{M}_{i}$ of $\bb{F}_{\lp}\Po$ is a further refinement of
the detailed picture of the submodule structure of permutation modules for classical groups, and may be of independent interest.

\section{Definitions and Notation}\label{nota}
\begin{defi}
A strongly regular graph (SRG) with parameters $(\tilde{v},\tilde{k},\tilde{\lambda},\tilde{\mu})$ is a $\tilde{k}$-regular graph on $\tilde{v}$ vertices such that:
\begin{enumerate}
\item Any two adjacent vertices have $\tilde{\lambda}$ neighbours;
\item Any two non-adjacent vertices have $\tilde{\mu}$ neighbours.
\end{enumerate}
\end{defi}

Let $q=p^{t}$ be a power of a prime, and let $\bb{F}_{q}$, $\bb{F}_{q^{2}}$ be finite fields of order $q$, and $q^{2}$ respectively. Let $V$ be either a vector space over $\bb{F}_{q}$ endowed with a non-degenerate symplectic form, a quadratic form, or a vector space over $\bb{F}_{q^{2}}$ carrying a non-degenerate Hermitian form. By $\tilde{q}$, we denote the size of the underlying field associated with $V$. We note that $\tilde{q}=q^{2}$ in the Hermitian case and is $q$ in the other cases.

Let $\Po$ be the set of all singular $1$-spaces in $V$. Given two distinct $\mathbf{v},\mathbf{u} \in \Po$, we say $\mathbf{v} \sim \mathbf{u}$ if and only if $\mathbf{v}$ and  $\mathbf{u}$ are orthogonal. Let $\Gamma(V)$ be the graph on $\Po$, whose adjacency is defined by the relation $\sim$.
Based on the parity of $\dim(V)$ and the geometry on $V$, we classify $\Gamma(V)$ into six families. We associate a parameter $h \in\{0,\frac{1}{2},1,\frac{3}{2},2\}$ with each family. Another parameter associated with each family is the dimension of the maximal totally isotropic subspace of $V$, denoted by $z$. The graph $\Gamma(V)$ is a graph on $(\tilde{q}^{z-1+h}+1)\dfrac{\tilde{q}^{z}-1}{\tilde{q}-1}$ vertices, where $\tilde{q}$ is the size of the underlying field. Any strongly regular $\Gamma(V)$ is isomorphic to one of the following graphs.   
\begin{enumerate}
\item When $V=\bb{F}_{q}^{2m}$ is a symplectic space (with $m\geq 2$), we denote $\Gamma(V)$ by $\Gamma_{s}(q,m)$. In this case $h=1$, $z=m$, and $\mathrm{Sp}(2m,q)\leq Aut(\Gamma_{s}(q,m))$. 
\item When $V= \bb{F}_{q}^{2m+1}$ is an endowed with a non-degenerate quadratic form  (with $m\geq 2$), we denote $\Gamma(V)$ by $\Gamma_{o}(q,m)$. In this case $h=1$, $z=m$, and $\mathrm{O}(2m,q) \leq Aut(\Gamma_{o}(q,m))$. 
\item When $V=\bb{F}_{q}^{2m}$  (with $m\geq 3$) is endowed with a non-degenerate elliptic quadratic form, we denote $\Gamma(V)$ by $\Gamma_{o-}(q,m)$. In this case $h=2$, $z=m-1$, and $\mathrm{O}^{-}(2m,q) \leq Aut(\Gamma_{o-}(q,m))$. 
 \item When $V=\bb{F}_{q}^{2m}$ (with $m\geq 3$) is endowed with a non-degenerate hyperbolic quadratic form, we denote $\Gamma(V)$ by $\Gamma_{o+}(q,m)$. In this case $h=0$, $z=m$, and $\mathrm{O}^{+}(2m,q) \leq Aut(\Gamma_{o+}(q,m))$.  
\item When $V=\bb{F}_{q^{2}}^{2m}$ (with $m\geq 2$) is endowed with a non-degenerate Hermitian quadratic form, we denote $\Gamma(V)$ by $\Gamma_{ue}(q,m)$. In this case $h=\frac{1}{2}$, $z=m$, and $\mathrm{U}(2m,q^2)\leq Aut(\Gamma_{ue}(q,m))$. 
 \item When $V=\bb{F}_{q^{2}}^{2m+1}$ (with $m\geq 2$) is endowed with a non-degenerate Hermitian quadratic form, we denote $\Gamma(V)$ by $\Gamma_{uo}(q,m)$. In this case $h=\frac{3}{2}$, $z=m$, and $\mathrm{U}(2m+1,q^2) \leq Aut(\Gamma_{ue}(q,m))$. 
\end{enumerate} 
From now on, we assume that $\Gamma(V)$ is one of the six graphs described above.
A \emph{polar graph} is a graph of the form $\Gamma(V)$.
By $\G$, we denote the group of form-preserving automorphisms of $V$. For example when $\Gamma(V)=\Gamma_{s}(q,m)$, we have $\G =\mathrm{Sp}(2m,q)$. 

We denote the number of $j$-dimensional subspaces of $\bb{F}_{q}^{d}$, by ${d \brack j}_{q}$. By standard counting arguments we have ${d \brack j}_{q} =\prod_{i=1}^{j} \frac{q^{d-i+1}-1}{q-1}$. Note that $q$ may be replaced by a variable $\mathfrak{z}$ to define ${d \brack j}_{\mathfrak{z}} :=\prod_{i=1}^{j} \frac{\mathfrak{z}^{d-i+1}-1}{\mathfrak{z}-1}$. 

By standard arguments (cf. \cite{BCN} \S 9.5), we can deduce the following result.
\begin{lemma}\label{par}
 The graph $\Gamma(V)$ is a strongly regular graph with parameters 
$$\left(
\begin{aligned}
v&:=(\tilde{q}^{z-1+h}+1){z \brack 1}_{\tilde{q}},\quad  k:=\tilde{q}{z-1 \brack 1}_{\tilde{q}}(\tilde{q}^{z-2+h}+1),\\
\lambda&:=(\tilde{q}-1)+\tilde{q}^{2}(\tilde{q}^{z-3+h}+1){z-2 \brack 1}_{\tilde{q}},\quad  \mu:=\frac{k}{\tilde{q}}
\end{aligned} \right).
$$  Here $\tilde{q}=q^{2}$ when $\Gamma(V)$ is either $\Gamma_{ue}(q,m)$ or $\Gamma_{uo}(q,m)$; and $\tilde{q}=q$ in other cases.  
\end{lemma}

Given a matrix $X \in M_{n \times m}(\bb{Z})$, let $\mathrm{Ab}(X)$ be the finite part of $\bb{Z}^{n}/X(\bb{Z}^{m})$. Now if we are given a prime $\lp$ and a positive integer $a$, by $e_{a}$ we denote the multiplicity of $\lp^{a}$ as an elementary divisor of the matrix $X$. Then $e_{0}$ is the $\lp$-rank of $X$ and for $a>0$, the multiplicity of $\lp^{a}$ as an elementary divisor of $\mathrm{Ab}(X)$ is $e_{a}$, by the structure theorem
for finitely generated abelian groups.

Throughout the paper, the adjacency and Laplacian matrices of $\Gamma(V)$ will be denoted by $A$ and $L$. The \emph{Smith group} of $\Gamma(V)$ is $\mathrm{Ab}(A)$ and the \emph{critical} group is $\mathrm{Ab}(L)$. Throughout the paper, $S$ and $K$ will denote the Smith and critical groups of $\Gamma(V)$ respectively. 

It follows from the definition of strongly regular graphs (cf. \cite{BH} Theorem $8.1.2$) that $A$ satisfies $A^{2}-(\mu-\lambda)A+(\mu-k)I=\mu J$. Here $J$ is the matrix of all ones.
By observing that $\mathbf{1}:=\sum\limits_{y \in \Po}y$ is an eigenvector for $A$, corresponding to the eigenvalue $k$, we have
$(A-kI)(A^{2}-(\mu-\lambda)A+(\mu-k)I)=0.$
Using this relation, the following Lemma can be derived using elementary linear algebra.
\begin{lemma}\label{rel}
Let $A$ be an adjacency matrix of $\Gamma(V)$, then $L=kI-A$ is the Laplacian. Let $r$ be the positive root of $\mathfrak{z}^{2}-(\mu-\lambda)\mathfrak{z}+(\mu-k)$, and $s$ the negative root. Let $t=k-r$, and $u=k-s$. Then the following hold.   
\begin{enumerate} 
\item We have $r=\tilde{q}^{z-1}-1$, $s=-(\tilde{q}^{z-2+h}+1)$,\\ $t=\ {z-1 \brack 1}_{\tilde{q}}(\tilde{q}^{z-1+h}+1)$,
 and $u=(\tilde{q}^{z-2+h}+1){z \brack 1}_{\tilde{q}}$.
\item $A$ has three eigenvalues, $(k,r,s)$ with multiplicities $(1,f,g)$. Where,

 $f=\dfrac{\tilde{q}^{h}(\tilde{q}^{z-2+h}+1)(\tilde{q}^{z}-1)}{(\tilde{q}-1)(\tilde{q}^{h-1}+1)}$, and $g=\dfrac{\tilde{q}(\tilde{q}^{z-1+h}+1)(\tilde{q}^{z-1}-1)}{(\tilde{q}-1)(\tilde{q}^{h-1}+1)}$.   
\item $L$ has eigenvalues $(0,t,u):=(0,k-r,k-s)$ with multiplicities $(1,f,g)$. 
\item $(\mathfrak{z}-k)(\mathfrak{z}-r)(\mathfrak{z}-s)$ is the minimal polynomial of $A$ and $(\mathfrak{z})(\mathfrak{z}-t)(\mathfrak{z}-u)$ is the minimal polynomial of $L$.
\item $(A-rI)(A-sI)=\mu J$.
\item $(L-tI)(L-uI)=-\mu J$.
\item $(\mathfrak{z}-kI)(\mathfrak{z}-rI)^{f}(\mathfrak{z}-sI)^{g}$ is the characteristic polynomial of $A$.
\item $\mathfrak{z}(\mathfrak{z}-tI)^{f}(\mathfrak{z}-uI)^{g}$ is the characteristic polynomial of $L$.
\end{enumerate} 
\end{lemma} 
As $A$ is a non-singular matrix, the order of the Smith group $S=|det(A)|$ and thus, 
$|S|=kr^{f}s^{g}$.
As a consequence of Kirchhoff's Matrix Tree Theorem (cf. \cite{Sta}), we have $|K|=\dfrac{t^{f}u^{g}}{v}$.

\section{Main Results}\label{main_results}
Our main results are presented in Theorems \ref{S} and \ref{K} below. A Polar graph as defined in the previous section is isomorphic to one of $\Gamma_{s}(q,m)$ (with $m \geq 2$), $\Gamma_{o}(q,m)$ (with $m \geq 2$), $\Gamma_{o-}(q,m)$ (with $m \geq 3$), $\Gamma_{o+}(q,m)$ (with $m \geq 3$), $\Gamma_{ue}(q,m)$ ( with $m \geq 2$), and $\Gamma_{uo}(q,m)$ (with $m \geq 2$). Theorem \ref{S}  describes the Smith groups of polar graphs and Theorem \ref{K} the critical groups. Corresponding to the six families of Polar graphs, the six tables following Theorem \ref{S} (respectively Theorem \ref{K}) encode the multiplicities of elementary divisors of the Smith (resp. critical) groups of these families of graphs. 

Given a prime $\lp$, the multiplicity $e_{i}$ of $\lp^{i}$ as an elementary divisor of $A$ (respectively $L$) is given in terms of parameters defined in the first two rows of the Tables \ref{so} to \ref{uoo} (respectively Tables \ref{sk} to \ref{uok}). The parameters $x,\ f,\ g$ defined in the first rows of the tables are dimensions of certain $\G$ representations. In particular, $f$ and $g$ are the multiplicities of the eigenvalues $r$ and $s$ (of $A$) respectively.  
The second rows define parameter $a,\ d,\ w$ (respectively $a,\ b,\ c,\ d$ in the case of $L$)  as $\lp$-adic valuations of certain divisors of eigenvalues $k,\ r, \ s$ of $A$ (respectively $t,\ u$ of $K$). We note that $v_{\lp}(r)=a+w$, $v_{\lp}(s)=d$, $v_{\lp}(k)=a+d$, $v_{\lp}(t)=a+c$, and $v_{\lp}(u)=b+d$. 

\begin{example}
The $4$th and $5$th rows of table \ref{sk} show that the $2$-elementary divisors $L$ when $\Gamma(V)=\Gamma_{s}(q,m)$ (with $q$ odd)

i) are $2^{0}$, $2^1$, $2^{d+1}$, $2^{d+b+1}$ with multiplicities $g+1$, $f-g-1$, $1$, and $g-1$ respectively, when $m$ is even;

ii) and are $2^{0}$, $2^1$, $2^{a+c}$, $2^{a+c+1}$ with multiplicities $g$, $1$, $f-g-1$, and $g$ respectively, when $m$ is odd.

Parameters $a,\ b,\ c,\ d, \ f$ and $g$ are as defined in the first two rows of table \ref{sk}.  
\end{example}

\begin{theorem}\label{S}
Let $V$ be a either a vector space over $\bb{F}_{q}$ endowed with a non-degenerate symplectic form, quadratic form, or a vector space over $\bb{F}_{q^{2}}$ carrying a non-degenerate Hermitian form.
 Further assume $\dim(V)\geq 4$ when $V$ is carrying a symplectic/Hermitian form, and $\dim(V)\geq 5$ when $V$ is endowed with a non-degenerate quadratic form.
Consider the graph $\Gamma(V)$, its Smith group $S$ and a prime $\lp \mid |S|$. 
If $\lp=p$, the $\lp$-part of $S$ is $\bb{Z}/q^{2}\bb{Z}$ when $\Gamma(V)$ is either $\Gamma_{ue}(q,m)$ or$\Gamma_{uo}(q,m)$, and $\bb{Z}/q\bb{Z}$ in other cases.
If $\lp \neq p$, the elementary divisors of $S$ are as described in Tables \ref{so}, \ref{oo}, \ref{omo}, \ref{opo}, \ref{ueo}, and \ref{uoo}. In these, $\delta_{ij}$ is $1$ if $i=j$ and $0$ otherwise.  
\end{theorem}
\footnotesize
\begin{table}[H]
\begin{tabular}{|c|c|p{7cm}|} 
\hline
\multicolumn{3}{|c|}{$(f,g):=\left(\ \frac{q(q^{m}-1)(q^{m-1}+1)}{2(q-1)},\ \frac{q(q^{m}+1)(q^{m-1}-1)}{2(q-1)}\right)$}  \\
\hline
\multicolumn{3}{|c|} {$(a,d,w):=\left(v_{\lp}({m-1 \brack 1}_{q}), \ v_{\lp}(q^{m-1}+1) , \ v_{\lp}(q-1)\right)$ }\\
\cline{1-3}
Prime & Arithmetic conditions & Non-zero divisor multiplicities \\\hline
\multirow{2}{*}{$\lp=2$} & $m$ is even and $q$ is odd  & 
$e_{0}=g+1$, $e_{w}=f-g-1$, and $e_{d+w}=g+1$. \\
\cline{2-3}
& $m$ is odd and $q$ is odd &  $e_{0}=g$, $e_{a}=1$, $e_{a+w}=f-g-1$, and $e_{a+w+1}=g+1$.\\
\cline{1-3}
\multirow{2}{*}{$\lp\neq2$} &  $d=0$  & $e_{0}=g+\delta_{a,0}$,  $e_{a}=\delta_{w,0}(f)+1+\delta_{a,0}(g)$, and $e_{a+w}=f+\delta_{w,0}$. \\
\cline{2-3}
&  $a=w=0$ & $e_{0}=f$, and $e_{d}=g+1$.\\
\cline{1-3}
\end{tabular}
\caption{Smith group of $\Gamma_{s}(q,m)$.}
\label{so}
\end{table}
 \normalsize
 See \S\ref{easy} (set $h=1$ and $z=m$) and \S\ref{s} for computation of the Smith group of $\Gamma_{s}(q,m)$.
 \footnotesize
 \begin{table}[H]

 \begin{tabular}{|c|c|p{8cm}|} 
\hline
\multicolumn{3}{|c|}{$(x,f,g):=\left(\dfrac{q^{2m}-q^{2}}{q^{2}-1},\ \frac{q(q^{m}-1)(q^{m-1}+1)}{2(q-1)},\ \frac{q(q^{m}+1)(q^{m-1}-1)}{2(q-1)}\right)$}\\
\hline
\multicolumn{3}{|c|}{$(a,d,w):=\left(v_{\lp}({m-1 \brack 1}_{q}), \ v_{\lp}(q^{m-1}+1) , \ v_{\lp}(q-1)\right)$}\\
\hline
Prime & Arithmetic conditions & Non-zero divisor multiplicities \\\hline
\multirow{2}{*}{$\lp=2$} & $m$ is even and $q$ is odd  & 
$e_{0}=x+1$, $e_{d}=g-x$, $e_{w}=f-x-1$, and $e_{d+w}=x+1$. \\
\cline{2-3}
& $m$ is odd and $q$ is odd & $e_{0}=x$, $e_{1}=g-x+\delta_{a,1}$, $e_{a}=1+\delta_{a,1}(g-x)$, $e_{a+w}= f-x-1$, and $e_{a+w+1}=x+1$.\\
\cline{1-3}
\multirow{2}{*}{$\lp\neq2$} &  $d=0$  & $e_{0}=g+\delta_{a,0}$,  $e_{a}=\delta_{w,0}(f)+1+\delta_{a,0}(g)$, and $e_{a+w}=f+\delta_{w,0}$. \\
\cline{2-3}
&  $a=w=0$ & $e_{0}=f$ and $e_{d}=g+1$.\\
\cline{1-3}
\end{tabular} 
\caption{Smith group of $\Gamma_{o}(q,m)$.}
\label{oo}
 \end{table}
 \normalsize
 See \S\ref{easy} (set $h=1$ and $z=m$) and \S\ref{o} for the computation of the Smith group of $\Gamma_{o}(q,m)$.
\footnotesize
 \begin{table}[H]
 \begin{tabular}{|c|c|p{7cm}|} 
\hline
\multicolumn{3}{|c|}{$(f,g):=\left( \frac{q^2(q^{m-1}-1)(q^{m-1}+1)}{(q^2-1)},\ \frac{q(q^{m}+1)(q^{m-2}-1)}{(q^2-1)}\right)$}\\
\hline
\multicolumn{3}{|c|}{$(a,d,w):=\left(v_{\lp}({m-2 \brack 1}_{q}), \ v_{\lp}(q^{m-1}+1) , \ v_{\lp}(q-1)\right)$}\\
\hline
Prime & Arithmetic conditions & Non-zero divisor multiplicities \\\hline
\multirow{2}{*}{$\lp=2$} & $m$ is even and $q$ odd.  & 
$e_{0}=g$, $e_{a}=1$, $e_{a+w}=f-g-1$, and $e_{a+d+w}=g+1$. \\
\cline{2-3}
& $m$ is odd and $q$ is odd. &  $e_{0}=g+1$, $e_{w}=f-g-1$, $e_{w+1}=g+1$.\\
\cline{1-3}
\multirow{3}{*}{$\lp\neq2$} & $q+1 \equiv 0 \pmod{\lp}$ and $m$ is even& $e_{0}=g$, $e_{a}=f-g$, and $e_{a+d}=g+1$ \\
\cline{2-3}
 &  $q \not\equiv -1 \pmod{\lp}$ and $d=0$  & $e_{0}=g+\delta_{a,0}$,  $e_{a}=\delta_{w,0}(f)+1+\delta_{a,0}(g)$, and $e_{a+w}=f+\delta_{w,0}$. \\
\cline{2-3}
&  $q \not\equiv -1 \pmod{\lp}$ and $a=w=0$ & $e_{0}=f$, and $e_{d}=g+1$.\\
\cline{1-3}
\end{tabular} 
\caption{Smith group of $\Gamma_{o-}(q,m)$.}
\label{omo}
 \end{table}
 \normalsize
See \S\ref{easy} (set $h=2$ and $z=m-1$) and \S\ref{om} for computation of the Smith group of $\Gamma_{o-}(q,m)$.
\footnotesize
\begin{table}[H]
 \begin{tabular}{|c|c|p{7cm}|}
\hline
\multicolumn{3}{|c|}{
$(f,g):=\left( \frac{q(q^{m}-1)(q^{m-2}+1)}{(q^2-1)},\ \frac{q^2(q^{m-1}+1)(q^{m-1}-1)}{(q^2-1)}\right)$}\\
\hline
\multicolumn{3}{|c|}{
$(a,d,w):=\left(v_{\lp}({m-1 \brack 1}_{q}), \ v_{\lp}(q^{m-2}+1) , \ v_{\lp}(q-1)\right)$}\\
\hline
Prime & Arithmetic conditions & Non-zero divisor multiplicities \\\hline
\multirow{2}{*}{$\lp=2$} & $m$ is even and $q$ odd.  & 
$e_{0}=f$, $e_{1}=g+1-f$, and $e_{w+1}=f$. \\
\cline{2-3}
& $m$ is odd and $q$ is odd. &  $e_{0}=f-1$, $e_{d}=g-f+1+\delta_{a,d}$, $e_{a}=\delta_{a,d}(g-f+1)+1$, and $e_{a+w+d}=f$.\\
\cline{1-3}
\multirow{3}{*}{$\lp\neq2$} & $q+1 \equiv 0 \pmod{\lp}$ and $m$ is odd& $e_{0}=f-1$, $e_{a}=1+\delta_{a,d}(g-f+1)$, $e_{d}=\delta_{a,d}+(g-f+1)$, and $e_{a+b}=f$. \\
\cline{2-3}
 &  $q \not\equiv -1 \pmod{\lp}$ and $d=0$  & $e_{0}=g+\delta_{a,0}$,  $e_{a}=\delta_{w,0}(f)+1+\delta_{a,0}(g)$, and $e_{a+w}=f+\delta_{w,0}$. \\
\cline{2-3}
&  $q \not\equiv -1 \pmod{\lp}$ and $a=w=0$ & $e_{0}=f$ and $e_{d}=g+1$.\\
\cline{1-3}
\end{tabular} 
\caption{Smith group of $\Gamma_{o+}(q,m)$.}
\label{opo}
 \end{table}
 \normalsize
See \S\ref{easy} (set $h=0$ and $z=m$) and \S\ref{op} for the computation of the Smith group of $\Gamma_{o+}(q,m)$.
\footnotesize
\begin{table}[H]
\begin{tabular}{|c|c|p{8cm}|}
\hline
\multicolumn{3}{|c|}{
$(x,f,g):= \left( \dfrac{(q^{2m}-1)(q^{2m-1}+1)}{(q^{2}-1)(q-1)}, \ \dfrac{q^{2}{m \brack 1}_{q^{2}}(q^{2m-3}+1)}{q+1},  \ \dfrac{q^{3}{m-1 \brack 1}_{q^{2}}(q^{2m-1}+1)}{q-1}\right)$}
\\
\hline
\multicolumn{3}{|c|}{
$(a,d,w):=\left(v_{\lp}({m-1 \brack 1}_{q^2}), \ v_{\lp}(q^{2m-3}+1),\ v_{\lp}(q^2-1) \right)$}\\
\hline
Prime & Arithmetic conditions & Non-zero divisor multiplicities \\\hline
\multirow{3}{*}{$\lp\mid q+1$} & $\lp \nmid m$ and $\lp\nmid m-1$ & $e_{0}=x+1$, $e_{w}=f-x-1+\delta_{w,d}(g-x)$, $e_{d}=\delta_{w,d}(f-x-1)+g-x$, and $e_{w+d}=x+1$. \\
\cline{2-3}
& $\lp \mid m-1$ & $e_{0}=x$, $e_{a}=1+\delta_{a,d}(g-x)$, $e_{d}=g-x+\delta_{a,d}$, $e_{w+a}=f-x-1$, and $e_{w+a+d}=x+1$.\\
\cline{2-3}
& $\lp \mid m$ & $e_{0}=x+1$, $e_{d}=g-x$, $e_{w}=f-x-1$, and $e_{w+d}=x+1$.\\
\cline{1-3}
\cline{1-3}
\multirow{2}{*}{$\lp\nmid q+1$} &  $d=0$  & $e_{0}=g+\delta_{a,0}$,  $e_{a}=\delta_{w,0}(f)+1+\delta_{a,0}(g)$, and $e_{a+w}=f+\delta_{w,0}$. \\
\cline{2-3}
&  $a=w=0$ & $e_{0}=f$ and $e_{d}=g+1$.\\
\cline{1-3}
\end{tabular}
\caption{Smith group of $\Gamma_{ue}(q,m)$.}
\label{ueo}
\end{table}
\normalsize
 See \S\ref{easy} (set $h=1/2$ and $z=m$) and \S\ref{ue} for the computation of the Smith group of $\Gamma_{ue}(q,m)$.
\footnotesize
\begin{table}[H]
\begin{tabular}{|c|c|p{8cm}|}
\hline
\multicolumn{3}{|c|}{
$(f,g):=\left(\dfrac{q^{3}{m \brack 1}_{q^{2}}(q^{2m-1}+1)}{q+1},\ \dfrac{q^{2}{m-1 \brack 1}_{q^{2}}(q^{2m-2}-1)}{q-1}\right)$} \\
\hline
\multicolumn{3}{|c|}{
$(a,d,w):=\left(v_{\lp}({m-1 \brack 1}_{q^2}), \ v_{\lp}(q^{2m+1}+1),\ v_{\lp}(q^2-1) \right)$}\\
\hline
Prime & Arithmetic conditions & Non-zero divisor multiplicities \\\hline
\multirow{2}{*}{$\lp\mid q+1$} & $\lp \nmid m$ & $e_{0}=g$, $e_{a}=1$, $e_{w+a}=f-g-1$, and $e_{w+a+d}=g+1$.\\
\cline{2-3}
& $\lp \mid m$ & $e_{0}=g+1$, $e_{w}=f-g-1$, and $e_{w+d}=g+1$.\\
\cline{1-3}
\multirow{2}{*}{$\lp\nmid q+1$} &  $d=0$  & $e_{0}=g+\delta_{a,0}$,  $e_{a}=\delta_{w,0}(f)+1+\delta_{a,0}(g)$, and $e_{a+w}=f+\delta_{w,0}$. \\
\cline{2-3}
&  $a=w=0$ & $e_{0}=f$ and $e_{d}=g+1$.\\
\cline{1-3}
 
\end{tabular}
\caption{Smith group of $\Gamma_{uo}(q,m)$.}
\label{uoo}
\end{table}
\normalsize
See \S\ref{easy} (set $h=3/2$ and $z=m$) and \S\ref{last} for the computation of the Smith group of $\Gamma_{uo}(q,m)$.
\begin{theorem}\label{K}
Let $V$ be a either a vector space over $\bb{F}_{q}$ endowed with a non-degenerate symplectic form, quadratic form, or a vector space over $\bb{F}_{q^{2}}$ carrying a non-degenerate Hermitian form.
 Further assume $\dim(V)\geq 4$ when $V$ is carrying a symplectic/Hermitian form, and $\dim(V)\geq 5$ when $V$ is endowed with a non-degenerate quadratic form.
Consider the graph $\Gamma(V)$, its critical group $K$ and a prime $\lp \mid |K|$. The $\lp$-elementary divisors of $K$ are as described in Tables \ref{sk}, \ref{ok}, \ref{omk}, \ref{opk}, \ref{uek}, and \ref{uok}. In these, $\delta_{ij}$ is $1$ if $i=j$ and $0$ otherwise.  
\end{theorem}
\footnotesize
\begin{table}[H]
\begin{tabular}{|c|c|p{7cm}|} 
\hline
\multicolumn{3}{|c|}{$(f,g):=\left(\ \frac{q(q^{m}-1)(q^{m-1}+1)}{2(q-1)},\ \frac{q(q^{m}+1)(q^{m-1}-1)}{2(q-1)}\right)$}  \\
\hline
\multicolumn{3}{|c|} {$(a,b,c,d):=\left(v_{\lp}({m-1 \brack 1}_{q}),\ v_{\lp}({m \brack 1}_{q}), \ v_{\lp}(q^{m}+1), \ v_{\lp}(q^{m-1}+1) \right)$}\\
\cline{1-3}
Prime & Arithmetic conditions & Non-zero divisor multiplicities \\\hline
\multirow{2}{*}{$\lp=2$} & $m$ is even and $q$ is odd  & 
$e_{0}=g+1$, $e_{1}=f-g-1$, $e_{d+1}=1$, and $e_{d+b+1}=g-1$. \\
\cline{2-3}
& $m$ is odd and $q$ is odd &  $e_{0}=g$, $e_{a}=1$,  $e_{a+c}=f-g-1$, and $e_{a+c+1}=g$.\\
\cline{1-3}
\multirow{2}{*}{$\lp\neq2$} &  $b=d=0$  & $e_{0}=g+\delta_{a,0}$,  $e_{a}=\delta_{c,0}(f-1)+1+\delta_{a,0}(g)$, and $e_{a+c}=f-1+\delta_{c,0}$. \\
\cline{2-3}
&  $a=c=0$ & $e_{0}=f+\delta_{d,0}$, $e_{d}=\delta_{b,0}(g)+1+\delta_{d,0}(f)$, and $e_{b+d}= g-1+\delta_{b,0}$\\
\cline{1-3}
\end{tabular}
\caption{Critical group of $\Gamma_{s}(q,m)$.}
\label{sk}
\end{table}
\normalsize
See \S\ref{easy} (set $h=1$ and $z=m$) and \S\ref{s} for computation of the critical group of 
\footnotesize
$\Gamma_{s}(q,m)$.
 \begin{table}[H]
\begin{tabular}{|c|c|p{7cm}|} 
\hline
\multicolumn{3}{|c|}{$(x,f,g):=\left(\dfrac{q^{2m}-q^{2}}{q^{2}-1},\ \frac{q(q^{m}-1)(q^{m-1}+1)}{2(q-1)},\ \frac{q(q^{m}+1)(q^{m-1}-1)}{2(q-1)}\right)$}  \\
\hline
\multicolumn{3}{|c|} {$(a,b,c,d):=\left(v_{\lp}({m-1 \brack 1}_{q}),\ v_{\lp}({m \brack 1}_{q}), \ v_{\lp}(q^{m}+1), \ v_{\lp}(q^{m-1}+1) \right)$}\\
\cline{1-3}
Prime & Arithmetic conditions & Non-zero divisor multiplicities \\\hline
\multirow{2}{*}{$\lp=2$} & $m$ is even and $q$ is odd & 
$e_{0}=x+1$, $e_{1}=f-x-1$, $e_{d+1}=1+\delta_{b,1}(g-x)$, $e_{b+d}=g-x+\delta_{b,1}$, and $e_{d+b+1}=x-1$. \\
\cline{2-3}
& $m$ is odd and $q$ is odd &  $e_{0}=x$, $e_{1}=g-x+\delta_{a,1}$, $e_{a}=\delta_{a,1}(g-x)+1$,  $e_{a+c}=f-x-1$, and $e_{a+c+1}=x$.\\
\cline{1-3}
\multirow{2}{*}{$\lp\neq2$} &  $b=d=0$  & $e_{0}=g+\delta_{a,0}$,  $e_{a}=\delta_{c,0}(f-1)+1+\delta_{a,0}(g)$, and $e_{a+c}=f-1+\delta_{c,0}$. \\
\cline{2-3}
&  $a=c=0$ & $e_{0}=f+\delta_{d,0}$, $e_{d}=\delta_{b,0}(g)+1+\delta_{d,0}(f)$, and $e_{b+d}= g-1+\delta_{b,0}$\\
\cline{1-3}
\end{tabular} 
\caption{Critical group of $\Gamma_{o}(q,m)$.}
\label{ok}
 \end{table}
 \normalsize
 See \S\ref{easy} (set $h=1$ and $z=m$) and \S\ref{o} for the computation of the critical group of $\Gamma_{o}(q,m)$.
\footnotesize
 \begin{table}[H]
 \begin{tabular}{|c|c|p{7cm}|} 
 \hline
\multicolumn{3}{|c|}{$(f,g):=\left( \frac{q^2(q^{m-1}-1)(q^{m-1}+1)}{(q^2-1)},\ \frac{q(q^{m}+1)(q^{m-2}-1)}{(q^2-1)}\right)$}\\
\hline
\multicolumn{3}{|c|}{$(a,b,c,d):=\left(v_{\lp}({m-2 \brack 1}_{q}),\ v_{\lp}({m-1 \brack 1}_{q}), \ v_{\lp}(q^{m}+1), \ v_{\lp}(q^{m-1}+1) \right)$}\\
\hline
\
Prime & Arithmetic conditions & Non-zero divisor multiplicities \\\hline
\multirow{2}{*}{$\lp=2$} & $m$ is even and $q$ odd.  & 
$e_{0}=g$, $e_{a}=1$, $e_{a+1}=f-g-1$, and $e_{a+d+1}=g$. \\
\cline{2-3}
& $m$ is odd and $q$ is odd. &  $e_{0}=g+1$, $e_{c}=f-g-1$, $e_{c+1}=1$, and $e_{b+c+1}=g-1$.\\
\cline{1-3}
\multirow{4}{*}{$\lp\neq2$} & $q+1 \equiv 0 \pmod{\lp}$ and $m$ is even& $e_{0}=g$, $e_{a}=f-g$, and $e_{a+d}=g$ \\
\cline{2-3}
& $q+1 \equiv 0 \pmod{\lp}$ and $m$ is odd & $e_{0}=g+1$, $e_{c}=f-g$, and $e_{b+c}=g-1$ \\
\cline{2-3}
 &  $q \not\equiv -1 \pmod{\lp}$ and $b=d=0$  & $e_{0}=g+\delta_{a,0}$,  $e_{a}=\delta_{c,0}(f-1)+1+\delta_{a,0}(g)$, and $e_{a+c}=f-1+\delta_{c,0}$. \\
\cline{2-3}
&  $q \not\equiv -1 \pmod{\lp}$ and $a=c=0$ & $e_{0}=f+\delta_{d,0}$, $e_{d}=\delta_{b,0}(g)+1+\delta_{d,0}(f)$, and $e_{b+d}= g-1+\delta_{b,0}$\\
\cline{1-3}
\end{tabular} 
\caption{Critical group of $\Gamma_{o-}(q,m)$.}
\label{omk}
 \end{table}
\normalsize
See \S\ref{easy} (set $h=2$ and $z=m-1$) and \S\ref{om} for computation of the \emph{critcal} group of $\Gamma_{o-}(q,m)$.
\footnotesize
\begin{table}[H]
 \begin{tabular}{|c|c|p{7cm}|} 
\hline
\multicolumn{3}{|c|}{
$(f,g):=\left( \frac{q(q^{m}-1)(q^{m-2}+1)}{(q^2-1)},\ \frac{q^2(q^{m-1}+1)(q^{m-1}-1)}{(q^2-1)}\right)$}\\
\hline
\multicolumn{3}{|c|}{
$(a,b,c,d):=\left(v_{\lp}({m-1 \brack 1}_{q}),\ v_{\lp}({m \brack 1}_{q}), \ v_{\lp}(q^{m-1}+1), \ v_{\lp}(q^{m-2}+1) \right)$}\\
\hline
Prime & Arithmetic conditions & Non-zero divisor multiplicities \\\hline
\multirow{2}{*}{$\lp=2$} & $m$ is even and $q$ odd.  & 
$e_{0}=f$, $e_{c+1}=\delta_{b,c}(g-f+1)+1$, $e_{b+1}=g-f+1+\delta_{b,c}$, and $e_{b+c+1}=f-2$. \\
\cline{2-3}
& $m$ is odd and $q$ is odd. &  $e_{0}=f-1$, $e_{a}=1+\delta_{a,d}(g+1-f)$, $e_{d}=g+1-f+\delta_{a,d}$,  and $e_{a+d+1}=f-1$.\\
\cline{1-3}
\multirow{4}{*}{$\lp\neq2$} & $q+1 \equiv 0 \pmod{\lp}$ and $m$ is odd& $e_{0}=f-1$, $e_{a}=1+\delta_{a,d}(g-f+1)$, $e_{d}=\delta_{a,d}+(g-f+1)$ and $e_{a+d}=f-1$. \\
\cline{2-3}
& $q+1 \equiv 0 \pmod{\lp}$ and $m$ is even & $e_{0}=f$, $e_{c}=1+\delta_{b,c}(g-f+1)$, $e_{b}=g-f+1+\delta_{b,c}$, and $e_{b+c}=f-2$.\\
\cline{2-3}
 &  $q \not\equiv -1 \pmod{\lp}$ and $b=d=0$  & $e_{0}=g+\delta_{a,0}$,  $e_{a}=\delta_{c,0}(f-1)+1+\delta_{a,0}(g)$, and $e_{a+c}=f-1+\delta_{c,0}$. \\
\cline{2-3}
&  $q \not\equiv -1 \pmod{\lp}$ and $a=c=0$ & $e_{0}=f+\delta_{d,0}$, $e_{d}=\delta_{b,0}(g)+1+\delta_{d,0}(f)$, and $e_{b+d}= g-1+\delta_{b,0}$\\
\cline{1-3}
\end{tabular} 
\caption{Critical group of $\Gamma_{o+}(q,m)$.}
\label{opk}
 \end{table}
\normalsize
See \S\ref{easy} (set $h=0$ and $z=m$) and \S\ref{op} for the computation of the critical group of $\Gamma_{o+}(q,m)$.
\footnotesize
\begin{table}[H]
\begin{tabular}{|c|c|p{8cm}|}
\hline
\multicolumn{3}{|c|}{
$(x,f,g):= \left( \dfrac{(q^{2m}-1)(q^{2m-1}+1)}{(q^{2}-1)(q-1)}, \ \dfrac{q^{2}{m \brack 1}_{q^{2}}(q^{2m-3}+1)}{q+1},  \ \dfrac{q^{3}{m-1 \brack 1}_{q^{2}}(q^{2m-1}+1)}{q-1}\right)$}
\\
\hline
\multicolumn{3}{|c|}{$(a,b,c,d):=\left(v_{\lp}({m-1 \brack 1}_{q^2}),\ v_{\lp}({m \brack 1}_{q^2}), \ v_{\lp}(q^{2m-1}+1), \ v_{\lp}(q^{2m-3}+1) \right)$}\\
\hline
Prime & Arithmetic conditions & Non-zero divisor multiplicities \\\hline
\multirow{3}{*}{$\lp\mid q+1$} & $\lp \nmid m$ and $\lp\nmid m-1$ & $e_{0}=x+1$, $e_{d}=f-x-1+\delta_{c,d}(g-x)$, $e_{c}=\delta_{c,d}(f-x-1)+g-x$, and $e_{c+d}=x$. \\
\cline{2-3}
& $\lp \mid m-1$ & $e_{0}=x$, $e_{a}=1+\delta_{a,c}(g-x)$, $e_{c}=\delta_{a,c}+(g-x)$, $e_{a+d}=f-x-1$, and $e_{c+d}=x$.\\
\cline{2-3}
& $\lp \mid m$ & $e_{0}=x+1$, $e_{d}=f-x-1$, $e_{b+d}=g-x+\delta_{b,d}$, $e_{2d}=1+\delta_{b,d}(g-x)$, and $e_{b+2d}=x-1$.\\
\cline{1-3}
\cline{1-3}
\multirow{2}{*}{$\lp\nmid q+1$} &  $b=d=0$  & $e_{0}=g+\delta_{a,0}$,  $e_{a}=\delta_{c,0}(f-1)+1+\delta_{a,0}(g)$, and $e_{a+c}=f-1+\delta_{c,0}$. \\
\cline{2-3}
&  $a=c=0$ & $e_{0}=f+\delta_{d,0}$, $e_{d}=\delta_{b,0}(g)+1+\delta_{d,0}(f)$, and $e_{b+d}= g-1+\delta_{b,0}$\\
\cline{1-3}
 
\end{tabular}
\caption{Critical group of $\Gamma_{ue}(q,m)$.}
\label{uek}
\end{table}
\normalsize
 See \S\ref{easy} (set $h=1/2$ and $z=m$) and \S\ref{ue} for the computation of the critical group of $\Gamma_{ue}(q,m)$.
\footnotesize
\begin{table}[H]
\begin{tabular}{|c|c|p{8cm}|}
\hline
\multicolumn{3}{|c|}{
$(f,g):=\left(\dfrac{q^{3}{m \brack 1}_{q^{2}}(q^{2m-1}+1)}{q+1},\ \dfrac{q^{2}{m-1 \brack 1}_{q^{2}}(q^{2m-2}-1)}{q-1}\right)$} \\
\hline
\multicolumn{3}{|c|}{
$(a,b,c,d):=\left(v_{\lp}({m-1 \brack 1}_{q^2}),\ v_{\lp}({m \brack 1}_{q^2}), \ v_{\lp}(q^{2m-1}+1), \ v_{\lp}(q^{2m+1}+1) \right)$}\\
\hline
Prime & Arithmetic conditions & Non-zero divisor multiplicities \\\hline
\multirow{2}{*}{$\lp\mid q+1$} & $\lp \nmid m$ & $e_{0}=g$, $e_{a}=1$, $e_{a+d}=f-g-1$, and $e_{a+c+d}=g$.\\
\cline{2-3}
& $\lp \mid m$ & $e_{0}=g+1$, $e_{d}=f-g-1$, $e_{2d}=1$, $e_{b+2d}=g-1$\\
\cline{1-3}
\multirow{2}{*}{$\lp\nmid q+1$} &  $b=d=0$  & $e_{0}=g+\delta_{a,0}$,  $e_{a}=\delta_{c,0}(f-1)+1+\delta_{a,0}(g)$, and $e_{a+c}=f-1+\delta_{c,0}$. \\
\cline{2-3}
&  $a=c=0$ & $e_{0}=f+\delta_{d,0}$, $e_{d}=\delta_{b,0}(g)+1+\delta_{d,0}(f)$, and $e_{b+d}= g-1+\delta_{b,0}$\\
\cline{1-3}
 
\end{tabular}
\caption{Critical group of $\Gamma_{uo}(q,m)$.}
\label{uok}
\end{table}
\normalsize
See \S\ref{easy} (set $h=3/2$ and $z=m$) and \S\ref{last} for the computation of the critical group of $\Gamma_{uo}(q,m)$.

\begin{remark}
We observe that the two families of polar graphs $\Gamma_{s}(q,m)$ and $\Gamma_{o}(q,m)$ are SRGs with the same parameters but different Smith and critical groups. This is an example where Smith and critical groups are distinguishing invariants for two families of isospectral graphs.  
\end{remark}

\section{Smith normal form}\label{snf}
Let $R$ be any PID and $T:R^{m} \to R^{n}$ be a linear transformation. By the structure theorem for finitely generated modules over PIDs, we have $\{\alpha_{i}\}_{i=1}^{s} \subset R \setminus \{0\}$ such that $\alpha_{i} \mid \alpha_{i+1}$ and 
$$\mathrm{coker}(T)\cong R^{n-s}\oplus \bigoplus\limits_{i=1}^{s} R /\alpha_{i}R.$$  
By $[T]$ we denote the matrix representation of $T$ with respect to standard bases. Then the above equation tells us that we can find $P \in \mathrm{GL}(R^{n})$, and $Q \in \mathrm{GL}(R^{m})$ such that 
$$P[T]Q=\left[
\begin{array}{c|c}
Y & 0_{(s \times n-s)} \\
\hline
0_{(m- \times s)} & 0_{(n-s \times n-s)}
\end{array}
\right],$$ 
where $Y=\mathrm{diag}(\alpha_{1}, \ldots ,\alpha_{s})$. The matrix $P[T]Q$ is called the Smith normal form of $T$.

Let $\lp \in R$ be a prime dividing $\alpha_{s}$. Given $j \in \bb{Z}_{\geq 0}$, we define $e_{j}(\lp):=|\{\alpha_{i}| \ v_{\lp}(\alpha_{i})=j \}|$. Now $e_{j}(\lp)$ is the multiplicities of $\lp ^{j}$ as an $\lp$-elementary divisors of $\mathrm{coker}(T)$. If $R=\bb{Z},$ then $e_{j}(\lp)$ is the multiplicity of $ \lp^{j}$ as an elementary divisor of the abelian group $\mathrm{coker}(T)$.

Let $R_{\lp}$ be the $\lp$-adic completion of $R$. We have 
$$R_{\lp}^{n}/T(R_{\lp}^{m})\cong R_{\lp}^{n-s}\oplus \bigoplus\limits_{j>0} \left(R_{\lp}/\lp^{j} R_{\lp}\right)^{e_{j}(\lp)}.$$
Define $M_{j}(T):=\{\mathfrak{z} \in R_{\lp}^{m}| \ T(\mathfrak{z}) \in \lp^{j}R_{\lp}^{n}\}$. For ease of notation, we denote $M_{j}(T)$ by $M_{j}$ and $e_{j}(\lp)$ by $e_{j}$. 
We have $R^{m}_{\lp}=M_{0}(T)\supset M_{1}(T) \supset \ldots \supset M_{n}(T) \supset \cdots$. 

Let $\bb{F}=R_{\lp}/ \lp R_{\lp}$.
If $M \subset R_{\lp}^{m}$ is a submodule, define $\ov{M}=(M+\lp R_{\lp}^{m})/ \lp R_{\lp}^{m}$. Then $\ov{M}$ is an $\bb{F}$-vector space.

The following Lemma follows from the structure theorem.
\begin{lemma}\label{eldivcal}
$e_{j}:=\dim(\ov{M_{j}(T)}/\ov{M_{j+1}(T)})$.
\end{lemma}  
So we have,
\begin{equation}\label{dim}
\dim(\ov{M_{j}(T)})-\dim(\ov{\ker(T)})=\sum\limits_{t\geq j} e_{t}.  
\end{equation}
Following the notation in $\S\ref{nota}$, given a matrix $C \in M_{n \times m}(\bb{Z})$, the finite part of $\bb{Z}^{n}/ C(\bb{Z}^{m})$ is denoted by $\mathrm{Ab}(C)$.
The following lemma, which will be applied frequently, is Lemma 3.1 of \cite{SD}. We include a short proof for the convenience of the reader.
\begin{lemma}\label{main}
Let $C$ be an $n\times m$ integer matrix. Fix a prime $\lp$ and let $d=v_{\lp}(\left|\mathrm{Ab}(C)\right|)$. Let $M_{i}:=M_{i}(C)$ be as defined above and $e_{i}:=e_{i}(\lp)$ be the $\lp$-elementary divisors of $C$. Suppose that we have two sequences of integers $0<t_{1}<t_{2} \ldots <t_{j}$ and $s_{1}>s_{2} \ldots >s_{j}>s_{j+1}=\dim(\ov{\ker(C)})$ satisfying the following conditions.
\begin{enumerate} [label=(\Alph*)]
\item $\dim(\ov{M_{t_{i}}})\geq s_{i}$, for all $1\leq i \leq j$. 
\item $d=\sum \limits_{i=1}^{j}(s_{i}-s_{i+1})t_{i}$.
\end{enumerate}
Then the following hold.
\begin{enumerate}[label=(\alph*)]
\item $e_{0}=m-s_{1}$.
\item $e_{t_{i}}=s_{i}-s_{i+1}$.
\item $e_{a}=0$ for $a \notin \{t_{1} \ldots t_{i}, \ldots t_{j}\}$.
\end{enumerate}
\end{lemma}
\begin{proof}
We have 
\begin{equation}\label{1}
d=\sum\limits_{i\geq 1}ie_{i} 
  \geq \sum \limits_{k=1}^{j-1}\left(\sum\limits_{t_{k}\leq i < t_{k+1}} ie_{i}\right)+\sum_{i \geq t_{j}}ie_{i}
 \geq \sum \limits_{k=1}^{j-1}\left(t_{k}\sum\limits_{t_{k}\leq i < t_{k+1}} e_{i}\right)+t_{j}\sum_{i \geq t_{j}}e_{i}.
\end{equation}

Application of equation \eqref{dim} given above the lemma yields 

\begin{equation*} 
 \sum \limits_{k=1}^{j-1}\left(t_{k}\sum\limits_{t_{k}\leq i < t_{k+1}} e_{i}\right)+t_{j}\sum_{i \geq t_{j}}e_{i}
 = \sum \limits_{k=1}^{j-1}\left(t_{k}(\dim(\ov{M_{t_{k}}})-\dim(\ov{M_{t_{k+1}}}))\right)+t_{j}\left(\dim(\ov{M_{t_{j}}})-\dim(\ov{\ker(C)})\right).
 \end{equation*}
 Now application of conditions (A) and (B) in the statement gives us
\begin{equation}\label{2}
 \sum \limits_{k=1}^{j-1}\left(t_{k}(\dim(\ov{M_{t_{k}}})-\dim(\ov{M_{t_{k+1}}}))\right)+t_{j}\left(\dim(\ov{M_{t_{j}}})-\dim(\ov{\ker(C)})\right)\geq \sum \limits_{i=1}^{j}(s_{i}-s_{i+1})t_{i}=d.
\end{equation}
So the inequalities \eqref{1} and \eqref{2} are in fact equations and thus the lemma follows.    
\end{proof}
The following result is $12.8.4$ of \cite{BH}.
\begin{lemma}\label{eigenval}
Let $C$ be an $n \times n$ integer matrix with an integer eigenvalue $\phi$ with geometric multiplicity $c$. Fix a prime $\lp$ dividing both $\left| \mathrm{Ab}\left(C\right)\right|$ and $\phi$, with $v_{\lp}(\phi)=d$. Then $\dim(\ov{M_{d}}(C)) \geq c$. 
\end{lemma}
\begin{proof}
Let $V_{\phi}$ be the eigenspace of $\bb{Q}_{\lp}^{n}$. Then $V_{\phi}\cap \bb{Z}_{\lp}^{n}$ is a pure $\bb{Z}_{\lp}$-submodule ($\bb{Z}_{\lp}$-direct summand) of
$\bb{Z}_{\lp}^{n}$ of rank $c$. It is clear that $V_{\phi}\cap \bb{Z}_{\lp}^{n} \subset M_{d}(C)$. As $V_{\phi}\cap \bb{Z}_{\lp}^{n}$ is pure, we have $\ov{V_{\phi}\cap \bb{Z}_{\lp}^{n}} \subset \ov{M_{d}(C)}$. 
\end{proof}
\section{Nilpotence of $A$ and $K$ modulo $\lp$.}\label{primes}
We recall from \S\ref{nota} that $\Gamma(V)$ is an SRG with parameters $(v,k,\lambda,\mu)$ specified in Lemma \ref{par}. Following notations fixed in \S\ref{nota}, $A$ will denote the adjacency matrix of $\Gamma(V)$ and $L=kI-A$ will denote the Laplacian matrix. By $J$, we denote the all-one matrix of same size as $A$. We also recall from Lemma \ref{rel} that $A$ has eigenvalues $k$, $r$, $s$, with multiplicities $1$, $f$ , and $g$ respectively; and that $L$ has eigenvalues $0$, $t=k-r$, $u=k-s$, with multiplicities $1$, $f$ , and $g$ respectively. The values of $r$, $s$, $t$, $u$, $f$, and $g$ are specified in Lemma \ref{rel}. We also observed that $|S|=kr^{f}s^{g}$ and that $|K|=\dfrac{t^{f}u^{g}}{v}$.    

Deducing from Lemma \ref{rel} that $k= -\tilde{q} s\dfrac{r}{\tilde{q}-1}$, we see that $\lp \mid |S|$ if and only if $\lp \mid \tilde{q}rs$. Since $\dfrac{tu}{v}$ is an integer, we see that $\lp \mid |K|$ if and only if $\lp \mid tu$. In the context of Lemma \ref{eigenval} and Lemma \ref{main}, it is useful to investigate the $\lp$-adic valuations of eigenvalues $r$, $s$ of $A$; and those of eigenvalues $t$ and $u$ of $L$. 

Given $X \in M_{n \times n}(\bb{Z})$, by $\ov{X}_{\lp}$ we denote the reduction of $X$ modulo $\lp$. The matrix $\ov{X}_{\lp}$ is nilpotent if and only if all eigenvalues of $X$ are divisible by $\lp$. Now the discussion in the above paragraph and enables us to make the following observations.

1) Since the $\tilde{q}$ is coprime to both $r$ and $s$, we see that $\ov{A}_{\lp}$ is nilpotent if and only if $\lp \mid r$ and $\lp \mid s$.

2) $\ov{L}_{\lp}$ is nilpotent if and only if $\lp \mid t$ and $\lp \mid u$.
      
The following Lemma completely classifies all the pairs $(\Gamma(V), \lp)$ for which $\ov{A}_{\lp}$ or $\ov{L}_{\lp}$ is nilpotent. 
\begin{lemma}
Consider the graph $\Gamma(V)$ and let $X$ be either the adjacency matrix or the Laplacian matrix of $\Gamma(V)$. Let $\lp$ be a prime and $\ov{X}_{\lp}$ be the reduction of $X$ $\pmod{\lp}$. Then conditions for nilpotence of $\ov{X}_{\lp}$ are encoded in Table \ref{nilconst}. 
\end{lemma} 
\small 
\begin{table}[H]
\begin{tabular}{|c|c|c|c|}
\hline
$\Gamma(V)$ & $\lp$ &  Arithmetic conditions & Nilpotence of $\ov{X}_{\lp}$\\
\hline
\multirow{2}{*}{$\Gamma_{s}(q,m)$ or $\Gamma_{o}(q,m)$} & $\lp=2$ & $q$ is odd & True \\
\cline{2-4}
& $\lp \neq 2$ & & False\\
\hline
\multirow{4}{*}{$\Gamma_{o-}(q,m)$}  & $\lp=2$ & $q$ is odd & True \\
\cline{2-4}
& $\lp \neq 2$ & ${\lp} \mid q+1$ and $m$ is even & True\\
\cline{2-4}
& $\lp \neq 2$ & ${\lp} \mid q+1$ and $m$ is odd & True for $\ov{L}_{\lp}$ and False for $\ov{A}_{\lp}$\\
\cline{2-4}
& $\lp \neq 2$ & ${\lp}\nmid q+1$ & False\\
\hline
\multirow{4}{*}{$\Gamma_{o+}(q,m)$}  & $\lp=2$ & $q$ is odd & True \\
\cline{2-4}
& $\lp \neq 2$ & ${\lp} \mid q+1$ and $m$ is odd & True\\
\cline{2-4}
& $\lp \neq 2$ & ${\lp} \mid q+1$ and $m$ is even & True for $\ov{L}_{\lp}$ and False for $\ov{A}_{\lp}$\\
\cline{2-4}
& $\lp \neq 2$ & ${\lp} \nmid q+1$ & False\\
\hline
\multirow{2}{*}{$\Gamma_{ue}(q,m)$ or $\Gamma_{uo}(q,m)$} & \multirow{2}{*}{$\lp$} & $\lp \mid q+1$ & True \\
\cline{3-4}
& & $\lp \nmid q+1$ & False\\
\hline
Any polar graph & $\lp=2$ & $q$ is even & False.\\
\hline
\end{tabular}
\caption{Conditions on $\lp$.}
\label{nilconst}
\end{table}
\normalsize
The proof follows by observing that $\ov{X}_{\lp}$ is nilpotent if and only if all three eigenvalues of $X$ are divisible by $\lp$.

Finding $\lp$-elementary divisors of $S$ and $K$ in the ``non-nilpotent'' cases is a bit easier. Lemma \ref{rel} gives us $(A-rI)(A-sI)=\mu J$ and  $(L-tI)(L-uI)=-\mu J$.
In this case Lemma \ref{eigenval} and the equations above help us construct two integer sequences satisfying the hypothesis of Lemma \ref{main}. We will do these computations in \S\ref{easy}.  

In the ``nilpotent'' case, we use representation theory of $\G$, the group of form preserving linear isomorphisms of $V$. Let us consider the case when $\lp$ is a ``nilpotent'' prime. 
We may treat $A$ and $L$ as elements of $\mathrm{End}_{\bb{Z}_{\lp}}(\bb{Z}_{\lp}\Po)$, where $\bb{Z}_{\lp}\Po$ is the free $\bb{Z}_{\lp}$ module with $\Po$ (vertex set of $\Gamma(V)$) as a basis. The action on $\Po$ by elements of the group $\G$ preserves adjacency and thus commutes with the actions of $A$ and $L$.
This implies that the $\bb{F}_{\lp}\Po$ subspaces $\ov{M}_{i}(A)$ and $\ov{M}_{i}(L)$ constructed as in \S\ref{snf} are also $\bb{F}_{\lp}\G$-submodules of the permutation module $\bb{F}_{\lp}\Po$. The action of $\G$ on $\Po$ has permutation rank $3$. The submodule structure of the permutation module $\bb{F}_{\lp}\Po$ has been determined in \cite{Lie1}, \cite{Lie2}, \cite{LST}, and \cite{ST} in {\it cross-characteristics}, that is, when
$\lp\nmid \tilde{q}$. We use these results along with Lemma \ref{main} to finish our computations.

\section{ When $\ov{A}_{\lp}$ and $\ov{L}_{\lp}$ are not nilpotent.}\label{easy}
In this section we deal with $\Gamma(V)$ and a prime $\lp$ such that $\ov{A}_{\lp}$ and $\ov{L}_{\lp}$ are not nilpotent. Table \ref{nilconst} can be used to look up all possible pairs $(V,\lp)$ such that $\ov{A}_{\lp}$ (equivalently $\ov{L}_{\lp}$) are not nilpotent.
\subsection{Elementary divisors of $S$}
The graph $\Gamma(V)$ is one of $\Gamma_{s}(q,m)$, $\Gamma_{o}(q,m)$, $\Gamma_{o-}(q,m)$, $\Gamma_{o+}(q,m)$, $\Gamma_{uo}(q,m)$ and $\Gamma_{ue}(q,m)$. Following the notation in Lemma \ref{rel}, we have $r={z-1 \brack 1}_{\tilde{q}}(\tilde{q}-1)$, $s=-(\tilde{q}^{z-2+h}+1)$, $\mu={z-1 \brack 1}_{\tilde{q}}(\tilde{q}^{z-2+h}+1)$, and $k=\tilde{q} \mu$. Here $\tilde{q}=q^2$  for $\Gamma_{uo}(q,m)$ and $\Gamma_{uo}(q,m)$; and $\tilde{q}=q$ for other graphs. 

If $\lp \mid |S|$, we saw in \S\ref{primes} that $\ov{A}_{\lp}$ is not nilpotent if and only if $\lp$ does not divide $r$ and $s$ simultaneously. Assume that $\lp$ does not divide $r$ and $s$ simultaneously, and that $\lp \mid |S|$. As $|S|=kr^{f}s^{g}$ and $k=\tilde{q}s \dfrac{r}{\tilde{q}-1}$, we see that $\lp$ divides exactly one of $\tilde{q}$, $r$, and $s$.

In this subsection, we identity $\ov{A}_{\lp}$ with $\ov{A}$ and $\ov{M}_{i}(A)$ with $\ov{M}_{i}$. 
\paragraph{Case 1: $\lp \mid r$ and $\lp \nmid s\tilde{q}$.}

 We set $v_{\lp}({z-1 \brack 1}_{\tilde{q}})=a$ and $v_{\lp}(\tilde{q}-1)=w$. Then $v_{\lp}(s)=0$, $v_{\lp}(k)=v_{\lp}(\mu)=a$, $v_{\lp}(r)=a+w$, and $v_{\lp}(|S|)=(a+w)f+a$. As $\lp \mid r$, one of $a$ and $w$ is necessarily non-zero.

By Lemma \ref{rel}, $(\mathfrak{z}-k)(\mathfrak{z}-s)^{g}(\mathfrak{z}-r)^{f}$ is the characteristic equation of $A$. Reducing modulo $\lp$, we see that $\mathfrak{z}^{f}(\mathfrak{z}-\ov{k})(\mathfrak{z}-\ov{s})^{g}$ is the characteristic polynomial of $\ov{A}$. By Lemma \ref{rel}, we observe that minimal polynomial of $\ov{L}$ divides $(\mathfrak{z}-\ov{k})(\mathfrak{z}-\ov{s})(\mathfrak{z})$, and thus all the Jordan blocks of $\ov{L}$ associated with $\ov{s}$ have size $1$. Therefore, the geometric multiplicity of $\ov{s}$ as an eigenvalue of $\ov{A}$ is $g$. We can now conclude that $\dim(\Ima(\ov{A}-\ov{s}I))= f+1$.

Lemma \ref{rel} give us $A(A-sI)=-r(A-sI)+\mu J$. Since $a=v_{\lp}(\mu) \leq v_{\lp}(r)$, we see that $\Ima(\ov{A}-\ov{s}I) \subset \ov{M}_{a}$. Thus $\dim(\ov{M}_{a}) \geq f+1$.

As $r$ is an eigenvalue of valuation $a+w$, Lemma \ref{eigenval} implies that $\dim(\ov{M}_{a+w}) \geq f$.

We apply Lemma \ref{main} to conclude the following.
\begin{enumerate}
\item Assume that $w=0$, then $a\neq 0$. As $A$ is non-singular, $\ker(A)=\{0\}$. Now by Lemma \ref{main}, setting $j=1$, $t_{1}=a$, $s_{1}=f+1 \leq \dim(\ov{M}_{a})$, $s_{2}=0=\dim(\ov{\ker(A)})$, we have
 $e_{a}=f+1$, $e_{0}=g$, and $e_{i}=0$ for all other $i$.
\item Assume $a=0$, then $w \neq 0$. As $A$ is non-singular, $\ker(A)=\{0\}$. Now by Lemma \ref{main}, setting $j=1$, $t_{1}=w$, $s_{1}=f \leq \dim(\ov{M}_{w})$, $s_{2}=0$, we have
 $e_{w}=f$, $e_{0}=g+1$ and $e_{i}=0$ for all other $i$. 
\item Assume $aw\neq 0$. As $A$ is non-singular, $\ker(A)=\{0\}$. Now by Lemma \ref{main}, setting $j=2$, $t_{1}=a+w,t_{2}=w$ $s_{1}=f+1 \leq \dim(\ov{M}_{w})$, $s_{2}=f \leq \dim(\ov{M}_{w})$, $s_{3}=0$, we have
 $e_{a+w}=1$, $e_{a}=f,$ $e_{0}=g$, and $e_{i}=0$ for all other $i$. 
\end{enumerate}

\paragraph{Case 2: $\lp \mid s$ and $\lp \nmid r\tilde{q}$.}

Set $v_{\lp}(s)=d$. As $\lp \nmid r$, we have $v_{\lp}(r)=0$. Then $v_{\lp}(k)=v_{\lp}(\mu)=d$, and $v_{\lp}(|S|)=dg+d$.

Lemma \ref{rel} gives us $A(A-rI)=-s(A-rI)+\mu J$. This shows that $\Ima(\ov{A}-\ov{r}I) \subset \ov{M}_{d}$. 
By Lemma \ref{rel}, $\mathfrak{z}(\mathfrak{z}-s)^{g}(\mathfrak{z}-r)^{f}$ is the characteristic polynomial of $A$. Reducing mod $\lp$, we see that $\ov{\mathfrak{z}}^{g+1}(\mathfrak{z}-\ov{r})^{f}$ is the characteristic polynomial of $\ov{A}$. Also Lemma \ref{rel}, we can deduce that $\mathfrak{z}(\mathfrak{z}-\ov{r})$ is the minimal polynomial of $\ov{A}$, and thus the geometric multiplicity of $\ov{r}$ as an eigenvalue of $\ov{A}$ is $f$. We can now conclude that $\dim(\Ima(\ov{A}-\ov{r}I))= g+1$.
Therefore $\dim(\ov{M}_{d})\geq g+1$. So by Lemma \ref{main}, setting $j=1$, $s_{1}=g+1$, $s_{2}=0$, and $t_{1}=d$, we have $e_{d}=g+1$, $e_{0}=f$ and $e_{i}=0$ for all other $i$. 

\paragraph{Case 3: $\lp \mid \tilde{q}$ and $\lp \nmid rs$.}

Set $v_{\lp}(\tilde{q})=\epsilon$. Then $v_{\lp}(k)=v_{\lp}(|S|)=a$. 
As $k$ is an eigenvalue of valuation $\epsilon$, Lemma \ref{eigenval} shows that $\dim(\ov{M}_{\epsilon}) \geq 1$. Thus by Lemma \ref{main}, we deduce that  $e_{\epsilon}=1$, $e_{0}=f+g$, and $e_{i}=0$ for all other $i$.

\subsection{Elementary divisors of $K$}
The graph $\Gamma(V)$ is one of $\Gamma_{s}(q,m)$, $\Gamma_{o}(q,m)$, $\Gamma_{o-}(q,m)$, $\Gamma_{o+}(q,m)$, $\Gamma_{uo}(q,m)$ and $\Gamma_{ue}(q,m)$. Following the notation in Lemma \ref{rel}, we have $\mu={z-1 \brack 1}_{\tilde{q}}(\tilde{q}^{z-2+h}+1)$, $t={z-1 \brack 1}_{\tilde{q}}(\tilde{q}^{z-1+h}+1)$,
  $u=(\tilde{q}^{z-2+h}+1){z \brack 1}_{\tilde{q}}$. and $v={z \brack 1}_{\tilde{q}}(\tilde{q}^{z-1+h}+1)$. Here $\tilde{q}=q^2$ for $\Gamma_{uo}(q,m)$ and $\Gamma_{uo}(q,m)$;  and $\tilde{q}=q$  for other graphs. 
  
If $\lp \mid |K|$, we saw in \ref{primes} that $\ov{L}_{\lp}$ is not nilpotent if and only if $\lp$ does not divide $t$ and $u$ simultaneously. Assume that $\lp$ does not divide $t$ and $u$ simultaneously, and that $\lp \mid |K|$. We recall that $|K|=t^{f}u^{g}/v$. 

In this subsection, we identity $\ov{L}_{\lp}$ with $\ov{L}$ and $\ov{M}_{i}(L)$ with $\ov{M}_{i}$. 

\paragraph{Case 1: $\lp \mid t$ and $\lp \nmid u$.}
In this case, $v_{\lp}(t)>0$ and $v_{\lp}(u)=0$. We set $v_{\lp}({z-1 \brack 1}_{\tilde{q}})=a$ and $v_{\lp}((\tilde{q}^{z-1+h}+1))=c$. Now, we have $v_{\lp}(t)=a+c$, $v_{\lp}(\mu)=a$, and $v_{\lp}(v)=c$. Since $|K|=t^{f}u^{g}/v$, we have $v_{\lp}(|K|)=(a+c)f-c$.
As $L$ is a matrix of nullity $1$, we have $\dim(\ov{\ker(L)})=1$.
As $t$ is an eigenvalue of $\lp$-valuation $a+c$ and geometric multiplicity $f$. So Lemma \ref{eigenval} implies that $\dim(\ov{M}_{a+c}) \geq f$.

Lemma \ref{rel} gives us $L(L-uI)=-t(L-uI)-\mu J$. So $\Ima(\ov{L}-\ov{u}I)\subset \ov{M}_{a}$. Again by Lemma \ref{rel}, $\mathfrak{z}(\mathfrak{z}-u)^{g}(\mathfrak{z}-t)^{f}$ is the characteristic polynomial of $L$. Reducing mod $\lp$, we see that $\mathfrak{z}^{f+1}(\mathfrak{z}-\ov{u})^{g}$ is the characteristic polynomial of $\ov{L}$. From Lemma \ref{rel} we deduce that minimal polynomial of $\ov{L}$ divides $\mathfrak{z}^2(\mathfrak{z}-\ov{u})$. Thus all the Jordan blocks of $\ov{L}$ associated with $\ov{u}$ are of size $1$. Therefore the geometric multiplicity of $\ov{u}$ as an eigenvalue of $\ov{L}$ is $g$. We can now conclude that $\dim(\Ima(\ov{L}-\ov{u}I))=f+1$, and thus $\dim(\ov{M}_{c}) \geq f+1$.

Using Lemma \ref{main}, we arrive at the following conclusions.
\begin{enumerate}
\item Assume that $a=0$, then $c \neq 0$. So by Lemma \ref{main}, setting $j=1$, $s_{1}=f$, $s_{2}=\dim(\ov{\ker(L)})=1$ and $t_{1}=c$, we have $e_{c}=f-1$, $e_{0}=g+1$, and $e_{i}=0$ for all other $i$.
\item Assume $c=0$, then $a \neq 0$. So by Lemma \ref{main}, setting $j=1$, $s_{1}=f+1$, $s_{2}=\dim(\ov{\ker(L)})=1$ and $t_{1}=a$, we have $e_{a}=f$, $e_{0}=g$, and $e_{i}=0$ for all other $i$.
\item Assume $ac \neq 0$. By Lemma \ref{main}, setting $j=2$, $s_{1}=f+1$, $s_{2}=f$, $s_{3}=\dim(\ov{\ker(L)})=1$ and $t_{1}=c$, $t_{2}=a+c$ we have $e_{0}=g$, $e_{a}=1$, $e_{a+c}=f-1$ and $e_{i}=0$ for all other $i$.   
\end{enumerate}
\paragraph{Case 2: $\lp \mid u$ and $\lp \nmid t$.}
In this case, $v_{\lp}(u)>0$, and $v_{\lp}(t)=0$. We set $v_{\lp}({z \brack 1}_{\tilde{q}})=b$ and $v_{\lp}((\tilde{q}^{z-2+h}+1))=d$.  We have $v_{\lp}(u)=b+d$, $v_{\lp}(\mu)=d$, and $v_{\lp}(v)=b$. Since $|K|=t^{f}u^{g}/v$, we have $v_{\lp}(|K|)=(b+d)g-b$.
As $L$ is a matrix of nullity $1$, we have $\dim(\ov{\ker(L)})=1$.
Since $u$ is an eigenvalue of valuation $d+b$ and geometric multiplicity $g$, Lemma \ref{eigenval} implies $\dim(\ov{M}_{d+b}) \geq g$.

By Lemma \ref{rel}, we have $L(L-tI)=-u(L-tI)-\mu J$. So $\Ima(\ov{L}-\ov{t}I)\subset \ov{M}_{d}$. Lemma \ref{rel} tells that $\mathfrak{z}(\mathfrak{z}-u)^{g}(\mathfrak{z}-t)^{f}$ is the characteristic polynomial of $L$.  Reducing modulo $\lp$, we see that $\ov{\mathfrak{z}}^{f+1}(\mathfrak{z}-\ov{t})^{g}$ is the characteristic polynomial of $\ov{L}$. From Lemma \ref{rel} we deduce that minimal polynomial of $\ov{L}$ divides $\mathfrak{z}^2(\mathfrak{z}-\ov{t})$. Thus all the Jordan blocks of $\ov{L}$ associated with $\ov{t}$ are of size $1$. Therefore the geometric multiplicity of $\ov{t}$ as an eigenvalue of $\ov{L}$ is $f$. We can now conclude $\dim(\Ima(\ov{L}-\ov{t}I))=g+1$ and $\dim(\ov{M}_{d}) \geq g+1$.

We now apply Lemma \ref{main} to conclude the following.
\begin{enumerate}
\item Assume $b=0$, then $d \neq 0$. So by Lemma \ref{main}, setting $j=1$, $s_{1}=g+1$, $s_{2}=\dim(\ov{\ker(L)})=1$ and $t_{1}=d$, we have $e_{0}=f$, $e_{d}=g$, and $e_{i}=0$ for all other $i$.
\item Assume $d=0$, then $b \neq 0$. So by Lemma \ref{main}, setting $j=1$, $s_{1}=g$, $s_{2}=\dim(\ov{\ker(L)})=1$ and $t_{1}=b$, we have $e_{0}=f+1$, $e_{b}=g-1$, and $e_{i}=0$ for all other $i$.
\item Assume $bd \neq 0$. Then by Lemma \ref{main}, setting $j=2$, $s_{1}=g+1$, $s_{2}=g$,  $s_{3}=\dim(\ov{\ker(L)})=1$ and $t_{1}=d$, $t_{2}=d+b$ we have $e_{0}=f$, $e_{d}=1$, $e_{d+b}=g-1$ and $e_{i}=0$ for all other $i$.   
\end{enumerate}

\section{When $\ov{A}_{\lp}$ and $\ov{L}_{\lp}$ are nilpotent.}\label{act}
Let $\lp$ be a prime and $\Gamma(V)$ be a polar graph such that $\ov{A}_{\lp}$ or $\ov{L}_{\lp}$ is nilpotent. In this case, we use representation theory of $\G$ to compute the $\lp$-elementary divisors of $S$ and $K$.  

The action of $\G$ on $\Gamma(V)$ commutes with $A$ and $L$. Thus the vector spaces $\ov{M}_{i}(A)$ and $\ov{M}_{i}(L)$ are in fact $\G$-submodules of $\bb{F}_{\lp}\Po$. We recall that the set $\Po$ which is the set of all singular $1$-spaces in $V$ is the vertex set of $\Gamma{V}$.

The action of $\G$ on $\Po$ is a rank $3$ permutation action. When $\lp$ is not the characteteristic of the field associated with the underlying vector space $V$, the submodule structure of $\bb{F}_{\lp}\Po$ is given in \cite{LST}, \cite{ST}, \cite{Lie1} and \cite{Lie2}. We use the submodule structures present in literature to determine $\ov{M}_{i}(A)$ and $\ov{M}_{i}(L)$ and consequently find the elementary divisors of $S$ and $K$.

We now define some submodules of $\bb{F}_{\lp}\Po$. These are some important submodules of $\bb{F}_{\lp}\Po$ defined in \cite{LST}, \cite{ST}, \cite{Lie1} and \cite{Lie2}.

1) Given any subspace $Z$ of $V$, we denote $[Z]$ to be the sum of all isotropic one-
dimensional subspace of $Z$. We denote $[V]$ by $\mathbf{1}$, henceforth known as the all-one vector.

2) Consider $A$ and $L$ to be elements of $\mathrm{End}(\bb{Q}_{\lp}\Po)$. Define $V_{r}=\ker(A-rI)=\ker(L-tI)$, and $V_{s}=\ker(A-sI)=\ker(L-uI)$. Then define $\ov{V}_{r}$ to be the subspace $ \ov{V_{r} \cap \bb{Z}_{\lp}\Po}$ of $\bb{F}_{\lp}\Po$, and $\ov{V}_{s}$ to be the subspace $ \ov{V_{s} \cap \bb{Z}_{\lp}\Po}$ of $\bb{F}_{\lp}\Po$. As $V_{r} \cap \bb{Z}_{\lp}\Po$ and $V_{s} \cap \bb{Z}_{\lp}\Po$ are pure submodules of $\bb{Z}_{\lp}\Po$, we have $\dim(\ov{V}_{r})=\dim(V_{r})=f$ and $\dim(\ov{V}_{s})=\dim(V_{s})=g$. 
 
3) We define $C$ to be the linear subspace of $\bb{F}_{\lp}\Po$ spanned by 
$$
\{[W]| W\ \text{is a maximal totally isotropic subspace of V}\}.
$$

4) We define $C^{'}$ to be the linear subspace of $\bb{F}_{\lp}\Po$ spanned by
$$
\{[W]-[W']| W, W' \ \text{are maximal totally isotropic subspaces of V}\}.
$$ 

5) We define $U$ to be $(J-\ov{A}_{\lp})(\bb{F}_{\lp}\Po)$, where $J$ is the matrix of all $1'$s.

6) We define $U'$ to be the subspace spanned by $\{(J-\ov{A}_{\lp})(\mathbf{v})-(J-\ov{A}_{\lp})(\mathbf{u})| \ \mathbf{v},\ \mathbf{u} \ \in \Po\}$.

7) Let $(\ ,\ )$ be the symmetric bilinear form on $\bb{F}_{\lp}\Po$ with $\Po$ as an orthonormal basis. If $Z$ is a subspace, then $Z^{\perp}$ denotes the orthogonal complement of $Z$ with respect to $(\ ,\ )$.

In the following Lemma, we collect some inclusion relations involving the modules defined above, $\ov{M}_{i}(A)$'s, and $\ov{M}_{i}(L)$'s.
\begin{lemma}\label{subair1}
Let $\lp$ be a prime and $\Gamma(V)$ be one of \{$\Gamma_{s}(q,m)$,\ $\Gamma_{o}(q,m)$,\ $\Gamma_{o-}(q,m)$,\ $\Gamma_{o+}(q,m)$,\ $\Gamma_{uo}(q,m)$,\ $\Gamma_{ue}(q,m)$\} such  that $\ov{A}_{\lp}$ or $\ov{L}_{\lp}$ is nilpotent. Also let $\tilde{q}$ be the size of the field associated with $V$. Then the following hold.
\begin{enumerate}
\item We have $\ov{V}_{s} \subset \ov{M}_{v_{\lp}(s)}(A)$, $\ov{V}_{r} \subset \ov{M}_{v_{\lp}(r)}(A)$, $\ov{V}_{s} \subset \ov{M}_{v_{\lp}(u)}(L)$, and $\ov{V}_{t} \subset \ov{M}_{v_{\lp}(t)}(L)$.
\item Given $\alpha=v_{\lp}({z-1 \brack 1}_{\tilde{q}})$, we have $C \subset \ov{M}_{\alpha}(A)$ and $C \subset \ov{M}_{\alpha}(L)$. 
\item Given $\beta=v_{\lp}(rs)$, we have $U \subset \ov{M}_{\beta}(A)$.  
\item Given $\gamma=v_{\lp}(tu)$, we have $U'\subset \ov{M}_{\gamma}(L)$
\item If $\lp\mid s$, then we have  $U \subset \ov{M}_{\delta}(L)$. Here $\delta=v_{\lp}(ts)$.
\end{enumerate}
\end{lemma}
\begin{proof}
1. The eigenspace associated with an eigenvalue $\alpha $ of $A$ is the same as the eigenspace associated with eigenvalue $k -\alpha$ of $L$. The proof of (1) now follows from the proof of Lemma \ref{eigenval}.

2.  Let $W$ be a maximal totally isotropic subspace of $V$ and let $\mathbf{v} \in \Po$. 
As $W$ is an isotropic subspace, if $\mathbf{v} \subset W$, then $\mathbf{v}$ is adjacent to every 
other $1$-space of $W$, a total of ${z \brack 1}_{\tilde{q}}-1$.
Assume that $\mathbf{v} \not\subset W$. Let $\mathbf{u} \in \Po$ and $\mathbf{u} \subset W$, then $\mathbf{v}$ is adjacent to $\mathbf{u}$ if and only if $\mathbf{u}$ is one of the ${z-1 \brack 1}_{\tilde{q}}$ $1$-dimensional subspaces of $\mathbf{v}^{\perp} \cap W$.  Here $\mathbf{v}^{\perp}$ is the orthogonal complement of $\mathbf{v}$ with respect to the form on $V$. 
Thus we have $A([W])={z-1 \brack 1}_{\tilde{q}}\mathbf{1}+r[W]$. Since $L=kI-A$, we also have $L([W])=-{z-1 \brack 1}_{\tilde{q}}\mathbf{1}+t[W]$. Using ${z-1 \brack 1}_{\tilde{q}} \mid t$, we arrive at 2. 

3. From Table \ref{nilconst}, we observe that nilpotence of $\ov{A}_{\lp}$ or $\ov{L}_{\lp}$ implies $\lp \mid q+1$. Therefore we have $\ov{A}-\ov{1/q}J\equiv\ov{A}+J \pmod{\lp}$. 
We note that $\Ima(J)=\bb{F}_{\lp}\mathbf{1}$, and thus $\Ima(\ov{A}+J)=\Ima(J-\ov{A})=U$. Using $AJ=kJ$, and $\mu=k/q$ and Lemma \ref{rel}, we have
$A(A-1/qJ-(r+s)I)=-rsI$. As $\lp \mid r$ and $\lp \mid s$, we can conclude 3.

4. This follows by using $LJ=0$, Lemma \ref{rel}, and calculations similar to those above. 

5. Let $\mathbf{v} \in \Po$ and let $W$ be any maximal totally isotropic subspace of $V$. From Lemma \ref{rel}, we have $L(L-(t+u)I)=-tuI-\mu J$. From the computations above, we have $L([W])=-{z-1 \brack 1}_{\tilde{q}}\mathbf{1}+t[W]$. These two observations together with $LJ=0$ give us
\begin{equation}\label{p1}
L\left(L(\mathbf{v})-(t+u)\mathbf{v}+J(\mathbf{v})-\dfrac{\mu}{{z-1 \brack 1}_{\tilde{q}}}[W]\right)=-tuL(\mathbf{v})-\dfrac{t\mu}{{z-1 \brack 1}_{\tilde{q}}}[W].
\end{equation}

Lemma \ref{par} and Lemma \ref{rel} show that $\dfrac{\mu}{{z-1 \brack 1}_{\tilde{q}}} =-s$. Since $\lp \mid s$, we have 
\begin{equation}\label{p2}
L(\mathbf{v})-(t+u)\mathbf{v}+J(\mathbf{v})-\dfrac{\mu}{{z-1 \brack 1}_{\tilde{q}}}[W]\equiv L(\mathbf{v})+J(\mathbf{v}) \pmod{\lp} \equiv J(\mathbf{v}) -A(\mathbf{v}) \pmod{\lp}.
\end{equation}
Now \eqref{p1} and \eqref{p2} yield 5.
\end{proof}

\section{$2$-elementary divisors of $S$ and $K$ when $\Gamma(V)=\Gamma_{s}(q,m)$.} \label{s} 
Given the graph $\Gamma_{s}(q,m)$ and a prime $\lp$, table \ref{nilconst} shows that $\ov{A}_{\lp}$ (equivalently $\ov{L}_{\lp}$) is nilpotent if and only if $\lp=2$ and $q$ is odd. 
In this section we will compute the $2$-elementary divisors of $S$ and $K$ when $\Gamma(V)=\Gamma_{s}(q,m)$ and $q$ is odd. We set $h=1$, and $z=m$ in Lemma \ref{rel} and Lemma \ref{par} to get the parameters for this graph. The graph $\Gamma_{s}(q,m)$ is an SRG with parameters $v={2m \brack 1}_{q}$, $k= q{m-1 \brack 1}_{q}(1+q^{m-1})$, $\lambda= {2m-2 \brack 1}_{q}-2$, and $\mu={2m-2 \brack 1}_{q}$. 

The adjacency matrix $A$ has eigenvalues $(k,r,s)=(k,q^{m-1}-1,-(1+q^{m-1}))$
with multiplicities $(1,f,g)=\left(1,\ \frac{q(q^{m}-1)(q^{m-1}+1)}{2(q-1)},\ \frac{q(q^{m}+1)(q^{m-1}-1)}{2(q-1)}\right)$. So the Laplacian $L$ has eigenvalues $(0,t,u)=(0,k-r,k-s)=\left(0,{m-1 \brack 1}_{q}(1+q^{m}), {m \brack 1}_{q}(1+q^{m-1})\right)$
with multiplicities $(1,f,g)$.

From now on in this section, we denote $\Gamma_{s}(q,m)$ by $\Gamma_{s}$.
\subsection{Submodule Structure}
We now recall from \S\ref{act} the definitions of $C$, $C'$ $U$, $U'$, $\ov{V}_{r}$ and $\ov{V}_{s}$ in the context of the graph $\Gamma_{s}$. In this case $\G=\mathrm{Sp}(2m,q)$.
From Theorem 2.13 and Remark 2.15 of \cite{LST}, we have the following result.
\begin{theorem} \label{lst}
The $\bb{F}_{2}\mathrm{Sp}(2m,q)$ submodule structure for $\bb{F}_{2}\Po$ is given by the following Hasse diagrams.

{\begin{minipage}[t]{0.48\textwidth}
m is even

$
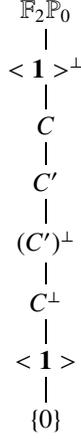
\xymatrix@=0.3cm{\bb{F}_{2}\Po \ar@{-}[d]\\ <\mathbf{1}>^{\perp} \ar@{-}[d] \\ C \ar@{-}[d] \\ C' \ar@{-}[d]\\ (C')^{\perp} \ar@{-}[d] \\ C^{\perp} \ar@{-}[d] \\ <\mathbf{1}> \ar@{-}[d] \\ \{0\}} $

\end{minipage}
\begin{minipage}[t]{0.48\textwidth}
m is odd

$
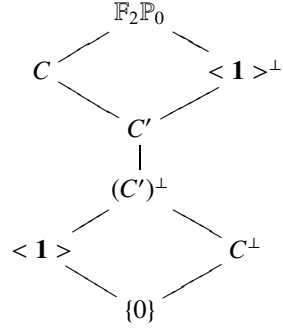
\xymatrix@=0.3cm{
 & \bb{F}_{2}\Po \\
 C \ar@{-}[ur] \ar@{-}[dr] &  & <\mathbf{1}>^{\perp}\ar@{-}[ul] \ar@{-}[dl]\\
 & C' \ar@{-}[d] & \\
 & (C')^{\perp} \ar@{-}[dr] \ar@{-}[dl] & \\
 <\mathbf{1}>  \ar@{-}[dr] & & C^{\perp}  \ar@{-}[dl]\\
 & \{0\} &
}$
\end{minipage}}
We have $U=(C')^{\perp}$, $U'=C^{\perp}$, $\dim(C)=f+1$, $\dim(C')=f$, $\dim((C^{'})^{\perp})=g+1$, and $\dim(C^{\perp})=g$.
\end{theorem}
\subsection{$2$-elementary divisors when $m$ is even}
\paragraph{Elementary divisors of $S$.}
We identify $\ov{M}_{i}(A)$ with $\ov{M}_{i}$ and $\ov{A}_{2}$ with $\ov{A}$.
Since $m$ is even, ${m-1 \brack 1}_{q}$ is odd. Thus have $v_{2}(q-1)=v_{2}(q^{m-1}-1)=v_{2}(r)$. We  set $v_{2}(q-1)=w$ and $v_{2}(s)=v_{2}(1+q^{m-1})=d$. So we have $v_{2}(k)=d=v_{2}(\mu)$. As $|S|=kr^{f}s^{g}$, we obtain $v_{2}(|S|)=d+wf+dg$.
Lemma \ref{subair1} implies that $U \subset \ov{M}_{d+w}$. By Theorem \ref{lst}, we conclude that $\dim(\ov{M}_{d+w})\geq \dim(U)=g+1$.
As $r$ is an eigenvalue of $2$-valuation $w$ and geometric multiplicity $f$, Lemma \ref{eigenval} implies that $\dim(\ov{M}_{w}) \geq f$.
So by Lemma \ref{main}, setting $j=2$, $s_{1}=f$, $s_{2}=g+1$, $s_{3}=\dim(\ov{\ker(A)})=0$, $t_{1}=w$, and $t_{2}=d+w$, we obtain
$e_{0}=g+1$, $e_{w}=f-g-1$, $e_{d+w}=g+1$, and $e_{i}=0$ for all other $i$.

\paragraph{Elementary divisors of $K$.} 
We identify $\ov{M}_{i}(L)$ with $\ov{M}_{i}$ and $\ov{L}_{2}$ with $\ov{L}$.
In this case ${m \brack 1}_{q}$ is even and thus $v_{2}(q^{m}-1)>1$. This implies that $v_{2}(q^{m}+1)=1$. Set $v_{2}({m \brack 1}_{q})=b$ and $v_{2}(s)=v_{2}(q^{m-1}+1)=d$. Then $v_{2}(t)=1$, $v_{2}(u)=b+d$, 
$v_{2}(k)=v_{2}(\mu)=b$, and $v_{2}(|K|)=f+(b+d)g-(b+1)$ (as $|K|=t^{f}u^{g}/v$).

Lemma \ref{subair1} gives us $U' \subset \ov{M}_{b+d+1}$ and $U \subset \ov{M}_{d+1}$. Now by Theorem \ref{lst}, we can see that 
$\dim(\ov{M}_{b+d+1})$  $\geq g$ and $\dim(\ov{M_{d+1}})\geq g+1$. As $t$ is an eigenvalue of valuation $1$ and geometric multiplicity $f$, Lemma \ref{eigenval} gives us $\dim(\ov{M}_{1}) \geq f$.

So by Lemma \ref{main}, setting $j=3$, $s_{1}=f$, $s_{2}=g+1$, $s_{3}=g$, $s_{4}=\dim(\ov{\ker(L)})=1$, and $t_{1}=1$, $t_{2}=d+1$, and $t_{3}=d+b+1$, we conclude that $e_{0}=g+1$, $e_{1}=f-g-1$, $e_{d+1}=1$, $e_{b+d+1}=g-1$, and $e_{i}=0$ for all other $i$.

\subsection{$2$-elementary divisors when $m$ is odd}
\paragraph{Elementary divisors of $S$.}
We identify $\ov{M}_{i}(A)$ with $\ov{M}_{i}$ and $\ov{A}_{2}$ with $\ov{A}$.
In this case $m$ is odd and thus ${m-1 \brack 1}_{q}$ is even, and therefore $v_{2}(q^{m-1}-1)>1$ and $v_{2}(q^{m-1}+1)=1$. We set $v_{2}({m-1 \brack 1}_{q})=a$ and $v_{2}(q-1)=w$. Now $v_{2}(r)=a+w$, $v_{2}(s)=v_{2}(q^{m-1}+1)=1$, and $v_{2}(k)=v_{2}(\mu)=a+1$. As $|S|=kr^{f}s^{g}$, we have $v_{2}(|S|)=(a+w)f+g+a+1$.
By Lemma \ref{subair1}, we have $C \subset \ov{M}_{a}$ and $U \subset \ov{M}_{a+w+1}$. Theorem \ref{lst} implies $\dim(\ov{M}_{a}) \geq f+1$ and $\dim(\ov{M}_{a+w+1})\geq g+1$.
 As $r$ is an eigenvalue of valuation $a+w$ and geometric multiplicity $f$, Lemma \ref{eigenval} implies that $\dim(\ov{M}_{a+w}) \geq f$.

So by Lemma \ref{main}, setting $j=3$, $s_{1}=f+1$, $s_{2}=f$, $s_{3}=g+1$, $s_{4}=\dim(\ov{\ker(A)})=0$, $t_{1}=a$, $t_{2}=a+w$, and $t_{3}=a+w+1$, we have
$e_{0}=g$, $e_{a}=1$, $e_{a+w}=f-g-1$, $e_{a+w+1}=g+1$, and $e_{i}=0$ for all other $i$. 

\paragraph{Elementary divisors of $K$.}
We identify $\ov{M}_{i}(L)$ with $\ov{M}_{i}$ and $\ov{L}_{2}$ with $\ov{L}$.
As ${m-1 \brack 1}_{q}$ is even we have $v_{2}(-1+q^{m-1})>1$, and thus $v_{2}(s)=v_{2}(1+q^{m-1})=1$. Set $v_{2}(q^{m}+1)=c$ and $v_{2}({m-1 \brack 1}_{q})=a$. We then have $v_{2}(t)=a+c$, $v_{2}(u)=1$, $v_{2}(v)=c$, $v_{2}(k)=v_{2}(\mu)=a+1$, and $v_{2}(|K|)=(a+c)f+g-c$.

By Lemma \ref{subair1}, we have $C \subset \ov{M}_{a}$ and $U \subset \ov{M}_{a+c+1}$.   Now application of Theorem \ref{lst} gives us $\dim(\ov{M}_{a})\geq f+1$ and $\dim(\ov{M}_{a+c+1})\geq g+1$.
 As $t$ is an eigenvalue of valuation $a+c$ and geometric multiplicity $f$, Lemma \ref{eigenval} implies that $\dim(\ov{M}_{a+c}) \geq f$.
 
So by Lemma \ref{main}, setting $j=3$, $s_{1}=f+1$, $s_{2}=f$, $s_{3}=g+1$, $s_{4}=\dim(\ov{\ker(L)})=1$, $t_{1}=a$, $t_{2}=a+c$, and $t_{3}=a+c+1$, we may conclude that
$e_{0}=g$, $e_{a}=1$, $e_{a+c}=f-g-1$, $e_{a+c+1}=g$, and $e_{i}=0$ for all other $i$. 
 
\section{$2$-elementary divisors of $S$ and $K$ when $\Gamma(V)=\Gamma_{o}(q,m)$.}\label{o}
Given the graph $\Gamma_{o}(q,m)$ and a prime $\lp$, table \ref{nilconst} shows that $\ov{A}_{\lp}$ (equivalently $\ov{L}_{\lp}$) is nilpotent if and only if $\lp=2$ and $q$ is odd. 
In this section we compute the $2$-elementary divisors of $S$ and $K$ when $\Gamma(V)=\Gamma_{o}(q,m)$. We set $h=1$, and $z=m$ in Lemma \ref{rel} and Lemma \ref{par} to get parameters for this graph.
The graph $\Gamma_{o}(q,m)$ is an SRG with parameters $v={2m \brack 1}_{q}$, $k= q{m-1 \brack 1}_{q}(1+q^{m-1})$, $\lambda= {2m-2 \brack 1}_{q}-2$, and $\mu={2m-2 \brack 1}_{q}$.

The Adjacency matrix $A$ has eigenvalues $(k,r,s)=(k,q^{m-1}-1,-(1+q^{m-1}))$
with multiplicities $(1,f,g)=\left(1,\ \frac{q(q^{m}-1)(q^{m-1}+1)}{2(q-1)},\ \frac{q(q^{m}+1)(q^{m-1}-1)}{2(q-1)}\right)$. So the Laplacian $L$ has eigenvalues $(0,t,u)=(0,k-r,k-s)=\left(0,{m-1 \brack 1}_{q}(1+q^{m}), {m \brack 1}_{q}(1+q^{m-1})\right)$
with multiplicities $(1,f,g)$.

From now on in this section, we denote $\Gamma_{o}(q,m)$ by $\Gamma_{o}$.
\subsection{Submodule structure}
We now recall from \S\ref{act} the definitions of $C$, $C'$ $U$, $U'$, $\ov{V}_{r}$ and $\ov{V}_{s}$ in the context of the graph $\Gamma_{o}$. In this case $\G=\mathrm{O}(2m+1,q)$.
From Theorem $1.1$, Corollary $7.5$, and Lemma $7.6$ of \cite{ST}, we have the following result.
\begin{theorem}\label{sto} The module $U'^{\perp}$ has a submodule $M$ containing $\ov{V}_{s}$ such that $\dim(M/\ov{V}_{s})=1$. The relative positions of $M$, $\ov{V}_{s}$, $C$, $\ov{V}_{r}$, $U^{\perp}$, $U'$, and $U$ in the $\bb{F}_{2}\mathrm{O}(2m+1,q)$ submodule lattice of $(U')^{\perp}$ are given by the following diagrams.

\small{\begin{minipage}[t][][c]{0.45 \textwidth}
$m$ is odd.\\
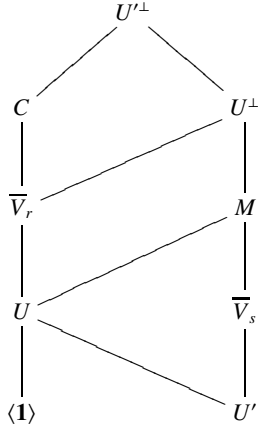
\xymatrix{&U'^{\perp}\ar@{-}[dr] \ar@{-}[dl]&  \\
C \ar@{-}[d]& & U^{\perp} \ar@{-}[dll] \ar@{-}[d]   \\
\ov{V}_{r} \ar@{-}[d] & & M \ar@{-}[dll] \ar@{-}[d] \\
U \ar@{-}[drr] \ar@{-}[d]  & & \ov{V}_{s} \ar@{-}[d]&  \\
\left\langle \mathbf{1} \right\rangle & & U' 
}
\end{minipage}
\begin{minipage}[t][][c]{0.45 \textwidth}
$m$ is even.\\
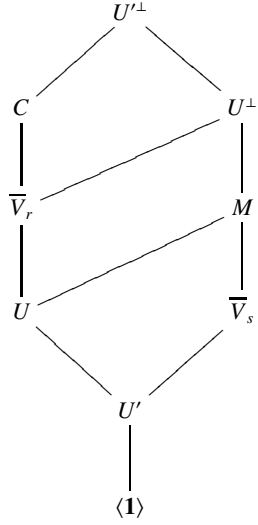
\xymatrix{&U'^{\perp}\ar@{-}[dr] \ar@{-}[dl]&  \\
C \ar@{-}[d]& & U^{\perp} \ar@{-}[dll] \ar@{-}[d]   \\
\ov{V}_{r} \ar@{-}[d] & & M \ar@{-}[dll] \ar@{-}[d] \\
U \ar@{-}[dr]   & & \ov{V}_{s} \ar@{-}[dl]&  \\
 &U' \ar@{-}[d] &  \\
 & \left\langle \mathbf{1} \right\rangle&
}
\end{minipage}}

\normalsize
Here $y:=\dim(\ov{V}_{s}/U')=\dfrac{(q^{m}-1)(q^{m})-q}{2(q+1)}$, $d:=\dim(\ov{V}_{r}/U)= \dfrac{(q^{m}+1)(q^{m}+q)}{2(q+1)}-1$, and $x:=\dim(U')=\dfrac{q^{2m}-q^{2}}{q^{2}-1}$,  $\dim(C/\ov{V}_{r})=1$, $\dim(M/ \ov{V}_{s})=1$, and $\dim(U/U')=1$.
\end{theorem}  
\begin{remark}\cite{ST} proves the above for $m \geq 3$. For $m=2$, we refer to Theorem $3.1$ of \cite{LST}.  We would like to address a typographical error present in Section $3$ of \cite{LST}. The definition of the submodule $\mathcal C$ of $k^{\mathcal{L}_2}$  should be changed from $\langle M \ | M \in \mathcal{L}_{2} \rangle _{k}$ to  $\langle \eta_{1,2}(M) \ | M \in \mathcal{L}_{1} \rangle_{k}$, and the related definition of $\mathcal C^+$ should be similarly corrected.
\end{remark}

\subsection{$2$-elementary divisors when $m$ is even}
\paragraph{Elementary divisors of $S$.}
We identify $\ov{M}_{i}(A)$ with $\ov{M}_{i}$ and $\ov{A}_{2}$ with $\ov{A}$.
Since ${m-1 \brack 1}_{q}$ is odd, we have $v_{2}(q-1)=v_{2}(q^{m-1}-1)$. Setting $v_{2}(q-1)=w$, $v_{2}(1+q^{m-1})=d$, we have $v_{2}(k)=d=v_{2}(\mu)$, $v_{2}(r)=w,$ $v_{2}(s)=d$, and $v_{2}(|S|)=d+dg+wf$. 
\subparagraph*{Case 1: Assume that $w=1$.}
In this case, as $v_{2}(q^{m-1}-1)=w=1$, we have $d=v_{2}(q^{m-1}+1)=v_{2}(q^{m-1}-1+2)>1$.
As $r,s$ are integer eigenvalues with non-zero $2$-valuations, Lemma \ref{subair1} gives us $\ov{V_{r}} \subset \ov{M_{1}}$ and $\ov{V_{s}}\subset \ov{M}_{1}$. Now by Theorem \ref{sto}, we have $\ov{V}_{r}+\ov{V}_{s}=U^{\perp}\subset \ov{M}_{1}$ and hence 
$\dim(\ov{M_{1}})\geq \dim(U^{\perp})=f+g+1-\dim(U)=f+g-x$.

By Lemma \ref{subair1}, $U \subset \ov{{M}_{d+1}}$. By Theorem \ref{sto} we have $\dim{U}=x+1$, and thus 

$\dim(\ov{M_{d+1}})\geq x+1$.

Now $s$ is an integer eigenvalue of geometric multiplicity $g$ and $2$-valuation $d$. Lemma \ref{subair1} gives us $\ov{V_{s}}\subset \ov{M_{d}}$. Also $U \subset \ov{M_{d+1}} \subset \ov{M_{d}}$. So by Theorem \ref{sto}, we have $\ov{V}_{s}+U=M \subset \ov{M}_{d}$ and hence $\dim(\ov{M_{d}})\geq \dim(M)=g+1$.
We have,
\begin{align*}
1(f+g-x-g-1)+d(g+1-(x+1))+(d+1)(x+1)=\\
f+d(g+1)=v_{2}(|S|).
\end{align*}
So by Lemma \ref{main}, setting $j=3$, $s_{1}=f+g-x$, $s_{2}=g+1$, $s_{3}=x+1$, $s_{4}=\dim(\ov{\ker(A)})=0$, $t_{1}=1$, $t_{2}=d$, and $t_{3}=d+1$, we may conclude that 
$e_{0}=x+1$, $e_{1}=f-x-1$, $e_{d}=g-x$, $e_{d+1}=x+1$, and $e_{i}=0$ for all other $i>0$.
\subparagraph*{Case 2: Assume that $w>1$.}
In this case, $d=1$.
As $r,s$ are integer eigenvalues with non-zero $2$-valuations, we have by Lemma \ref{subair1}, $\ov{V_{r}} \subset \ov{M_{1}}$ and $\ov{V_{s}}\subset \ov{M}_{1}$. Thus by Theorem \ref{sto}, $U^{\perp}\subset \ov{M}_{1}$. Hence $\dim(\ov{M_{1}})\geq f+g+1-x-1$. 

By Lemma \ref{subair1}, we have $U \subset \ov{{M}_{w+1}}$. Thus $\dim(\ov{M_{w+1}})\geq x+1$ and. We have
\begin{align*}
1(f+g-x-f)+w(f-(x+1))+(w+1)(x+1)=\\
wf+g+1=v_{2}(|S|).
\end{align*}
So by Lemma \ref{main}, setting $j=3$, $s_{1}=f+g-x$, $s_{2}=f$, $s_{3}=x+1$, $s_{4}=\dim(\ov{\ker(A)})=0$, $t_{1}=1$, $t_{2}=w$, and $t_{3}=w+1$, we may conclude that 
$e_{0}=x+1$, $e_{1}=g-x$, $e_{w}=f-x-1$, $e_{w+1}=x+1$, and $e_{i}=0$ for all other $i>0$.

\paragraph{Elementary divisors of $K$.}
We identify $\ov{M}_{i}(L)$ with $\ov{M}_{i}$ and $\ov{L}_{2}$ with $\ov{L}$.
In this case, $m$ is even. Since ${m \brack 1}_{q}$ is even, we have $v_{2}(q^{m}-1)>1$ and $v_{2}(q^{m}+1)=1$. We set $v_{2}({m \brack 1}_{q})=b$ and $v_{2}(s)=v_{2}(q^{m-1}+1)=d$. So we have $v_{2}(t)=1$, $v_{2}(u)=d+b$, 
$v_{2}(k)=v_{2}(\mu)=d$, and $v_{2}(|K|)=f+(b+d)g-(b+1)$.

As $A \equiv L \pmod 2$, from Smith group computation above we conclude that $\dim(\ov{M_{1}})\geq f+g-x$.
As $u$ is an eigenvalue of valuation $d+b$, Lemma \ref{subair1} implies $\ov{V_{s}} \subset \ov{M_{d+b}}$, and thus $\dim(\ov{M}_{d+b})\geq g$. 
By Lemma \ref{subair1}, we have $U' \subset \ov{M}_{d+b+1}$ and $U \subset \ov{M}_{d+1}$. Therefore $\dim(\ov{M}_{d+b+1})\geq x$, by Theorem \ref{sto}.

Since $\ov{V_{s}} \subset \ov{M}_{d+b} \subset \ov{M}_{d+1}$ and $U \subset \ov{M}_{d+1}$, by Theorem \ref{sto} we have $\ov{M}_{d+1} \subset M$. We may now conclude that $\dim(\ov{M}_{d+1})\geq g+1$.

Now,
\begin{align*}
1(f+g-x-g-1)+(d+1)(g+1-g)+(d+b)(g-x)+(d+b+1)(x-1)=\\
f+(d+b)g-(b+1)=v_{2}(|K|).
\end{align*}
We apply Lemma \ref{main} to conclude the following.
\begin{itemize}
\item If $b>1$, set $j=4$, $s_{1}=f+g-x$, $s_{2}=g+1$, $s_{3}=g$, $s_{4}=x$ $s_{5}=\dim(\ov{\ker(L)})=1$, $t_{1}=1$, $t_{2}=d+1$, $t_{3}=d+b$, $t_{4}=d+b+1$, then by Lemma \ref{main}, we have 
$e_{0}=x+1$, $e_{1}=f-x-1$, $e_{d+1}=1$, $e_{d+b}=g-x$, $e_{d+b+1}=x-1$, and $e_{i}=0$ for all other $i$.
\item If $b=1$, set $j=3$, $s_{1}=f+g-x$, $s_{2}=g+1$, $s_{3}=x$, $s_{4}=\dim(\ov{\ker(L)})=1$, $t_{1}=1$, $t_{2}=d+1$, $t_{3}=d+b+1$, then by Lemma \ref{main}, we have
$e_{0}=x+1$, $e_{1}=f-x-1$,  $e_{1+d}=g+1-x$, $e_{d+b+1}=x-1$, and $e_{i}=0$ for all other $i$.
\end{itemize}

\subsection{$2$-elementary divisors of $S$ and $K$, when $m$ is odd.}
\paragraph{Elementary divisors of $S$.} 
We identify $\ov{M}_{i}(A)$ with $\ov{M}_{i}$ and $\ov{A}_{2}$ with $\ov{A}$.
In this case $m$ is odd.
As $m$ is odd, ${m-1 \brack 1}_{q}$ is even and thus $v_{2}({m-1 \brack 1}_{q})>0$. Set $v_{2}({m-1 \brack 1}_{q})=a$ and $v_{2}(q-1)=w$. So we have $v_{2}(r)=a+w$, $v_{2}(s)=1$, $v_{2}(k)=v_{2}(\mu)=a+1$, and $v_{2}(|S|)=(a+w)f+g+(a+1)$.
By Lemma \ref{subair1}, $\dim(\overline{{M}_{a}}) \geq f+1$ (since $\dim(C)=f+1$ by Theorem \ref{sto}).
As $s$ is an integer eigenvalue of multiplicity $g$ with $2$-valuation $1$, Lemma \ref{subair1} implies $\ov{V_{s}} \subset \overline{M_{1}} $. Since $C \subset\ov{M_{a}} \subset \ov{M_{1}}$ aswell , Theorem \ref{sto} implies $\ov{M_{1}} \supset \ov{V}_{s}+C= (U')^{\perp} $. So $\dim(\ov{M_{1}}) \geq \dim(U'^{\perp})=f+g+1-\dim(U')=f+g+1-x$.
As $r$ is an integer eigenvalue of multiplicity $f$ with $2$-valuation $a+w$, Lemma \ref{subair1} implies $\dim(\overline{M_{a+w}})\geq f$.
By Lemma \ref{subair1}, we also have $U \subset \ov{{M}_{a+w+1}}$. From Theorem \ref{sto}, we conclude $\dim(\ov{M_{a+w+1}})\geq x+1$.

We have
\begin{align*}
1(f+g+1-x-(f+1))+a(f+1-f)+(a+w)(f-(x+1))+(a+b+1)(x+1) \\
=g+a+((a+w)f)+1=v_{2}(|S|).
\end{align*}
We may conclude the following from Lemma \ref{main}.
\begin{itemize}
\item If $a>1$, setting $j=4$, $s_{1}=f+g+1-x$, $s_{2}=f+1$, $s_{3}=f$, $s_{4}=x+1$, $s_{5}=\dim(\ov{\ker(A)})=0$, $t_{1}=1$, $t_{2}=a$, $t_{3}=a+w$, and $t_{4}=a+w+1$, by Lemma \ref{main} we get $e_{0}=x$, $e_{1}=g-x$, $e_{a}=1$, $e_{a+w}= f-x-1$,  $e_{a+w+1}=x+1$ and $e_{i}=0$ for all other $i>0$.
\item If $a=1$, setting $j=3$, $s_{1}=f+g+1-x$, $s_{2}=f$, $s_{3}=x+1$, $s_{4}=\dim(\ov{\ker(A)})=0$, $t_{1}=1$, $t_{2}=a+w$, and $t_{3}=a+w+1$, by Lemma \ref{main}, $e_{0}=x$, $e_{1}=g+1-x$, $e_{a+w}= f-x-1$, $e_{a+w+1}=x+1$, and $e_{i}=0$ for all other $i>0$.
\end{itemize}

\paragraph{Elementary divisors of $K$.} 
We identify $\ov{M}_{i}(L)$ with $\ov{M}_{i}$ and $\ov{L}_{2}$ with $\ov{L}$.
In this case, $m$ is odd.
 As ${m-1 \brack 1}_{q}$ is even, we have $v_{2}(-1+q^{m-1})>1$ and thus $v_{2}(s)=v_{2}(1+q^{m-1})=1$. Set $v_{2}(q^{m}+1)=c$ and $v_{2}({m-1 \brack 1}_{q})=a$. We have $v_{2}(t)=a+c$, $v_{2}(u)=1$, $v_{2}(v)=c$, $v_{2}(k)=v_{2}(\mu)=a+1$, and $v_{2}(|K|)=(a+c)f+g-c$.
 
By Lemma \ref{subair1}, $C \subset \ov{M}_{a}$ and $U \subset \ov{M}_{a+c+1}$. Thus we have $\dim(\ov{M}_{a}) \geq f+1$ and $\dim(\ov{M}_{a+c+1})\geq x+1$, by Theorem \ref{sto}.
As $u$ is an integer eigenvalue of geometric multiplicity $g$ with $2$-valuation $1$, Lemma \ref{eigenval} implies $\overline{M_{1}}\supset \ov{V_{s}}$. Since $C \subset\ov{M_{a}} \subset \ov{M_{1}}$, Theorem \ref{sto} implies $\ov{M_{1}}\supset (U')^{\perp}$. Thus $\dim(\ov{M_{1}}) \geq f+g+1-x$.
As $t$ is an integer eigenvalue of multiplicity $g$ with $2$-valuation $a+b$, Lemma \ref{subair1} implies $\overline{M_{a+c}}\supset \ov{V_{r}}$. So $\dim(\ov{M_{a+c}}) \geq f$.
 
We have
\begin{align*}
1(f+g+1-x-(f+1))+a(f+1-f)
+(a+c)(f-(x+1))\\+(a+c+1)(x+1-1) \\
=g+(a+c)f-b=v_{2}(|K|).
\end{align*}
Using Lemma \ref{main}, we conclude the following.
\begin{itemize}
\item If $a>1$, setting $j=4$, $s_{1}=f+g+1-x$, $s_{2}=f+1$, $s_{3}=f$, $s_{4}=x+1$, $s_{5}=\dim(\ov{\ker(L)})=1$, $t_{1}=1$, $t_{2}=1$, $t_{3}=a+c$, and $t_{4}=a+c+1$, by Lemma \ref{main} we have $e_{0}=x$, $e_{1}=g-x$, $e_{a}=1$, $e_{a+c}=f-x-1$, $e_{a+c+1}=x$, and $e_{i}=0$ for all other $i$.
\item If $a=1$, by similar arguments we can deduce that $e_{0}=x$, $e_{1}=g-x+1$, $e_{a+c}=f-x-1$, $e_{a+c+1}=x$, and $e_{i}=0$ for all other $i$.  
\end{itemize}
\section{$\lp$-elementary divisors of $S$ and $K$ when $\Gamma(V)=\Gamma_{o-}(q,m)$, and $\lp \mid q+1$.}\label{om}
Given the graph $\Gamma_{o-}(q,m)$ and a prime $\lp$, table \ref{nilconst} shows that $\ov{A}_{\lp}$ is nilpotent if and only if either i) $\lp=2$ and $q$ is odd; or ii) $\lp$ is odd with $\lp \mid q+1$ and $m$ is even. We also have $\ov{L}_{\lp}$ is nilpotent if and only if $\lp \mid q+1$.

In this section we compute the $\lp$-elementary divisors of $S$ and $K$ when $\Gamma(V)= \Gamma_{o-}(q,m)$ and $(\lp, q,m)$ satisfy the arithmetic conditions given above.  

From now on in this section, we denote $\Gamma_{o-}(q,m)$ by $\Gamma_{o-}$, and $\lp$ is a prime that meets the description in the previous paragraph.
We set $h=2$ and $z=m-1$ in Lemma \ref{rel} and Lemma \ref{par} to get the parameters for this graph. Thus we get that $\Gamma_{o-}(V)$ is an SRG with parameters 
$v={m-1 \brack 1}_{q}(q^{m}+1)$, $k=q{m-2 \brack 1}_{q}(q^{m-1}+1)$,
$\lambda=q{m-2 \brack 1}_{q}(q^{m-1}+1)-1-q^{2m-3}$, and
$\mu={m-2 \brack 1}_{q}(q^{m-1}+1)$.

The eigenvalues of the adjacency matrix $A$ are $(k,r,s)=(k,q^{m-2}-1,-(1+q^{m-1})),$
with multiplicities $(1,f,g)=\left(1,\ \frac{q^{2}(q^{m-1}-1)(q^{m-1}+1)}{(q^{2}-1)},\ \frac{q(q^{m}+1)(q^{m-2}-1)}{(q^{2}-1)}\right)$.

So the Laplacian $L$ has eigenvalues $(0,t,u)=\left(0,{m-2 \brack 1}_{q}(1+q^{m}), {m-1 \brack 1}_{q}(1+q^{m-1})\right)$
with multiplicities $(1,f,g)$.

\subsection{Submodule structure}
We now recall from \S\ref{act} the definitions of $C$, $C'$ $U$, $U'$, $\ov{V}_{r}$ and $\ov{V}_{s}$ in the context of the graph $\Gamma_{o-}$. In this case $\G=\mathrm{O^-}(2m,q)$.
By Corollary $8.5$, Lemma $8.7$, Lemma $8.8$, Lemma $8.9$, Corollary $8.10$, and Corollary $8.11$ of \cite{ST}, we have the following result.
\begin{theorem}\label{sto-}
Given a prime $\lp$ with $\lp\mid q+1$.
The relative positions of $C$, $\ov{V}_{r}$, $\ov{V}_{s}$, $U'$, $U$, and $\left\langle \mathbf{1} \right\rangle$ in the $\bb{F}_{\lp}\mathrm{O^-}(2m,q)$ submodule structure of $C$ are in the following diagrams. We have $\dim(C)=f+1$.

\small{\begin{minipage}[t][][c]{0.40\textwidth}
$\lp=2$ and $m$ is even

\xymatrix{
&C \ar@{-}[d]&\\
& \ov{V}_{r} \ar@{-}[d] &\\
& U \ar@{-}[dr] \ar@{-}[dl] &\\
\left\langle \mathbf{1} \right\rangle &  &U'=\ov{V}_{s}}
\end{minipage}
\begin{minipage}[t][][c]{0.40\textwidth}
$\lp=2$ and $m$ is odd

\xymatrix{
&C \ar@{-}[d]&\\
& \ov{V}_{r} \ar@{-}[d] &\\
& U \ar@{-}[d]  &\\
 &  U'=\ov{V}_{s} \ar@{-}[d] & \\
 & \left\langle \mathbf{1} \right\rangle & }
\end{minipage}

\begin{minipage}[t][][c]{0.40\textwidth}
$\lp\neq 2$, $\lp \mid q+1$, and $m$ is even. 

\xymatrix{
&C \ar@{-}[d] \ar@{-}[dr]&\\
& U \ar@{-}[d] \ar@{-}[dr] & \ov{V}_{r} \ar@{-}[d]\\
& \left\langle \mathbf{1} \right\rangle  & U'=\ov{V}_{s} 
 }
\end{minipage}
\begin{minipage}[t][][c]{0.40\textwidth}
$\lp\neq 2$, $\lp \mid q+1$, and $m$ is odd. 

\xymatrix{
& C \ar@{-}[dl] \ar@{-}[dr]&\\
 U  \ar@{-}[dr] &  & \ov{V}_{r} \ar@{-}[dl]\\
   & U'=\ov{V}_{s} \ar@{-}[d] &\\
& \left\langle \mathbf{1} \right\rangle &
 }
\end{minipage}}
\end{theorem}
\subsection{$2$-elementary divisors of $S$ and $K$ when $m$ is odd.}
\paragraph{Elementary divisors of $S$.}
We identify $\ov{M}_{i}(A)$ with $\ov{M}_{i}$ and $\ov{A}_{2}$ with $\ov{A}$.
Since $m$ is odd, ${m-2 \brack 1}_{q}$ is an odd number. So we have 
$v_{2}(r)=v_{2}(q^{m-2}-1)=v_{2}(q-1)$. As ${m-1 \brack 1}_{q}$ is an even number, $v_{2}(q^{m-1}-1) >1$, and thus $v_{2}(q^{m-1}+1)=1=v_{2}(s)$. We also have $v_{2}(k)=v_{2}(\mu)=1$. Setting $v_{2}(q-1)=w$, we have $v_{2}(r)=w$ and $v_{2}(|S|)=wf+g+1$.

By Lemma \ref{subair1}, we have $U \subset \ov{M}_{w+1}$. So by Theorem \ref{sto-}, we get $\dim(\ov{M}_{w+1}) \geq g+1$.
Since $r$ is an integer eigenvalue with valuation $w$, Lemma \ref{eigenval} implies $\dim(\ov{M}_{w}) \geq f$.

Now we have $(w)(f-(g+1))+(w+1)(g+1)=wf+g+1=v_{2}(|S|)$. So by Lemma \ref{main}, setting $j=2$, $t_{1}=w$, $t_{2}=w+1$, $s_{1}=f$, $s_{2}=g+1$, $s_{3}=\dim(\ov{\ker(A)})=0$, we conclude that $e_{0}=g+1$, $e_{w}=f-g-1$, $e_{w+1}=g+1$, and $e_{i}=0$ for all other $i$.
 
\paragraph{Elementary divisors of $K$.}
We identify $\ov{M}_{i}(L)$ with $\ov{M}_{i}$ and $\ov{L}_{2}$ with $\ov{L}$.
Set $v_{2}(q^{m}+1)=c$ and $v_{2}({m-1 \brack 1}_{q})=b$. Since $m$ is odd, we have $v_{2}(s)=v_{2}(q^{m-1}+1)=1$. So $v_{2}(t)=c$, $v_{2}(u)=b+1$, $v_{2}(v)=c+b$, $v_{2}(\mu)=1$, and $v_{2}(|K|)=cf+(b+1)g-(c+b)$.

As $t$ is an eigenvalue of valuation $f$, by Lemma \ref{eigenval}, we have 
$\dim(\ov{M_{c}}) \geq f$.

Lemma \ref{subair1} implies that $U \subset \ov{M}_{c+1}$ and $U' \subset \ov{M}_{b+c+1}$. Therefore $\dim(\ov{M_{b+c+1}}) \geq g$ and $\dim(\ov{M_{c+1}})\geq g+1$, by Theorem \ref{sto-}.

Now, $c(f-(g+1))+(c+1)(g+1-g)+(b+c+1)(g-1)=cf+(b+1)g-b-c$.
So by Lemma \ref{main}, setting $j=3$, $s_{1}=f$, $s_{2}=g+1$, $s_{3}=g$, $s_{4}=\dim(\ov{\ker(L)})=1$, $t_{1}=c$, $t_{2}=c+1$, and $t_{3}=b+c+1$, we may conclude that
$e_{0}=g+1$, $e_{c}=f-g-1$, $e_{c+1}=1$, $e_{b+c+1}=g-1$ and $e_{i}=0$ for all other $i$.

\subsection{$2$-elementary divisors of $S$ and $K$ when $m$ is even.}
\paragraph{Elementary divisors of $S$.}
We identify $\ov{M}_{i}(A)$ with $\ov{M}_{i}$ and $\ov{A}_{2}$ with $\ov{A}$.
As $m$ is even, ${m-2 \brack 1}_{q}$ is even and ${m-1 \brack 1}_{q}$ is odd. Set $v_{2}({m-2 \brack 1}_{q})=a$, $v_{2}(q^{m-1}-1)=v_{2}(q-1)=w$ and $v_{2}(q^{m-1}+1)=d$. Then $v_{2}(r)=a+w$, $v_{2}(s)=d$, $v_{2}(k)=v_{2}(\mu)=a+d$, and $v_{2}(|S|)=(a+w)f+dg+a+d$.

Lemma \ref{subair1}, implies that $C \subset \ov{M}_{a}$. Thus by Theorem \ref{sto-} we have $\dim(\ov{M}_{a})\geq f+1$.

As $r$ is an eigenvalue of valuation $a+w$, Lemma \ref{eigenval} implies  $\dim(\ov{M}_{a+w})\geq f$.
By Lemma \ref{subair1}, $U \subset  \ov{M_{a+d+w}}$. Thus by Theorem \ref{sto-}, $\dim(M_{a+d+w})\geq g+1$.

We have $a(f+1 -f)+(a+w)(f-(g+1))+(a+d+w)(g+1)=(a+w)f+dg+a+d$.
So by Lemma \ref{main}, setting $j=3$, $s_{1}=f+1$, $s_{2}=f$, $s_{3}=g+1$, $s_{4}=\dim(\ov{\ker(A)})=0$, $t_{1}=a$, $t_{2}=a+w$, and $t_{3}=a+d+w$, we have $e_{0}=g$, $e_{a}=1, e_{a+w}=f-g-1, e_{a+d+w}=g+1$ and $e_{i}=0$ for all other $i$.

\paragraph{Elementary divisors of $K$.}
We identify $\ov{M}_{i}(L)$ with $\ov{M}_{i}$ and $\ov{L}_{2}$ with $\ov{L}$.
As $m$ is even, ${m-2 \brack 1}_{q}$ is even and ${m-1 \brack 1}_{q}$ is odd.
Set $v_{2}({m-2 \brack 1}_{q})=a$ and $v_{2}(s)=v_{2}(q^{m-1}+1)=d$. As ${m \brack 1}_{q}$ is even, we have $v_{2}(q^{m}-1)>1$ and $v_{2}(q^{m}+1)=1$. Since ${m-1 \brack 1}_{q}$ is odd, we have $v_{2}(v)=1$. We have $v_{2}(t)=a+1$, $v_{2}(u)=d$, and $v_{2}(\mu)=a+d$, and $v_{2}(|K|)=(a+1)f+dg-1$.

By Lemma \ref{subair1}, we have $C \subset \ov{M}_{a}$ and $U \subset \ov{M}_{a+d+1}$. Therefore by Theorem \ref{sto-}, we have $\dim(\ov{M}_{a}) \geq f+1$ and $\dim(\ov{M_{a+d+1}}) \geq g+1$.
As $t$ is an integer eigenvalue of $L$ with valuation $a+1$, Lemma \ref{eigenval} implies $\dim(\ov{M}_{a+1})\geq f$.

We have $a(f+1-f)+(a+1)(f-(g+1))+(a+d+1)(g+1-1)=(a+1)f+dg-1=v_{2}(|K|)$.
Therefore by Lemma \ref{main}, setting $j=3$, $s_{1}=f+1$, $s_{2}=f$, $s_{3}=g+1$, $s_{4}=\dim(\ov{\ker(L)})=1$, $t_{1}=a$, $t_{2}=a+1$, and $t_{3}=a+d+1$,
we have $e_{0}=g$, $e_{a}=1$, $e_{a+1}=f-g- 1$, $e_{a+d+1}=g$, and $e_{i}=0$ for all other $i\neq 0$.
\subsection{$\lp$-elementary divisors of $S$ and $K$ when $m$ is even, $\lp \neq 2$ and $\lp \mid q+1$.}
\paragraph{Elementary divisors of $S$.}
We identify $\ov{M}_{i}(A)$ with $\ov{M}_{i}$ and $\ov{A}_{\lp}$ with $\ov{A}$.
In this case $\lp$ is an odd prime dividing $q+1$ and $m$ is even. Thus $v_{\lp}({m-2 \brack 1}_{q})=v_{2}(r)$. 
Set
$v_{\lp}({m-2 \brack 1}_{q})=a$ and $v_{\lp}(s)=v_{\lp}(q^{m-1}+1)=d$. Then $v_{p}(r)=a$, $v_{\lp}(k)=a+d=v_{\lp}(\mu)$, and $v_{\lp}(|S|)=af+dg+a+d$.

By Lemma \ref{subair1}, we have $C \subset \ov{M_{a}}$ and  $U \subset \ov{M}_{a+d}$. Thus $\dim(\ov{M}_{a})\geq f+1$ and $\dim(\ov{M}_{a+d}) \geq g+1$, by Theorem \ref{sto-}.

We have $a(f+1-(g+1))+(a+d)(g+1)=v_{\lp}(|S|)$.
So by Lemma \ref{main}, setting $j=2$, $s_{1}=f+1$, $s_{2}=g+1$, $s_{3}=\dim(\ov{\ker(A)})=0$, $t_{1}=a$, $t_{2}=a+d$, we conclude $e_{0}=g$ $e_{a}=f-g$, $e_{a+d}=g+1$, and $e_{i}=0$ for all other $i$.

\paragraph{Elementary divisors of $K$.}
We identify $\ov{M}_{i}(L)$ with $\ov{M}_{i}$ and $\ov{L}_{\lp}$ with $\ov{L}$.
In this case $\lp$ is an odd prime with $m$ even. We set
$v_{\lp}({m-2 \brack 1}_{q})=a$, and $v_{\lp}(q^{m-1}+1)=d$. As  $q \equiv -1 \pmod \lp$, we have $v_{\lp}(v)=v_{\lp}\left(\left({m-1 \brack 1}_{q}(q^{m}+1)\right)\right)=0$. Thus $v_{\lp}(t)=a$, $v_{\lp}(u)=d$, and $v_{\lp}(|K|)=af+dg$. 

Lemma \ref{subair1} gives us $C \subset \ov{M_{a}}$ and $U \subset \ov{M}_{a+d}$. Now Theorem \ref{sto-} gives us $\dim(\ov{M_{a}}) \geq f+1$ and $\dim(\ov{M_{a+d}}) \geq g+1$.

We have $a(f+1-(g+1))+(a+d)(g+1-1)=af+dg=$.
So by Lemma \ref{main}, setting $j=2$, $s_{1}=f+1$, $s_{2}=g+1$, $s_{3}=\dim(\ov{\ker(A)})=1$, $t_{1}=a$, $t_{2}=a+d$, we have $e_{0}=g$,  $e_{a}=f-g$, $e_{a+d}=g$, and $e_{i}=0$ for all other $i\neq 0$.

\subsection{$\lp$-elementary divisors of $K$ when $m$ is odd, $\lp \neq 2$ and $\lp \mid q+1$.\\}
In this case we have $q \equiv -1 \pmod{\lp}$ and thus  $r \equiv s \equiv -2 \pmod{\lp}$ and thus $\lp \nmid |S|$. However we see that $\lp\mid t$ and $\lp \mid u$ and thus $\lp \mid |K|$. In this section we compute the $\lp$-elementary divisors of $K$.
\paragraph{Elementary divisors of $K$.}
We identify $\ov{M}_{i}(L)$ with $\ov{M}_{i}$ and $\ov{L}_{\lp}$ with $\ov{L}$.
Set $v_{\lp}(q^{m}+1)=c$ and $v_{\lp}({m-1 \brack 1}_{q})=b$. Since $m$ is odd, we have $v_{\lp}(s)=v_{\lp}(q^{m-1}+1)=0$. So $v_{\lp}(t)=c$, $v_{\lp}(u)=b$, $v_{\lp}(v)=c+b$, $v_{\lp}(\mu)=0$, and $v_{\lp}(|K|)=cf+bg-(c+b)$.

Lemma \ref{subair1} gives us $U' \subset \ov{M}_{b+c}$ and $\ov{V}_{r} \subset \ov{M_{c}}$. Therefore $\dim(\ov{M_{b+c}}) \geq g$ and $\dim(\ov{M_{c}})\geq f$, by Theorem \ref{sto-}.
Now we have $v_{\lp}(|K|)=cf+bg-c-b= c(f-g)+(b+c)(g-1)$.

So by Lemma \ref{main}, setting $j=2$, $s_{1}=f$, $s_{2}=g+1$, $s_{3}=\dim(\ov{\ker(L)})=1$, $t_{1}=c$, $t_{2}=b+c$, we may conclude that
$e_{0}=g+1$, $e_{c}=f-g$, $e_{b+c}=g-1$, and $e_{i}=0$ for all other $i$.

\section{$\lp$-elementary divisors of $S$ and $K$ when $\Gamma(V)=\Gamma_{o+}(q,m)$, and $\lp \mid q+1$}\label{op}
Given the graph $\Gamma_{o+}(q,m)$ and a prime $\lp$, table \ref{nilconst} shows that $\ov{A}_{\lp}$ is nilpotent if and only if either i) $\lp=2$ and $q$ is odd; ii) or $\lp$ is odd with $\lp \mid q+1$ and $m$ is odd. 
Also $\ov{L}_{\lp}$ is nilpotent if and only if $\lp \mid q+1$
In this section we compute the $\lp$-elementary divisors of $S$ and $K$ when $\Gamma(V)= \Gamma_{o+}(q,m)$ and $(\lp, q,m)$ satisfy the arithmetic conditions given above.

From now on in this section, we denote $\Gamma_{o+}(q,m)$ by $\Gamma_{o+}$, and $\lp$ is a prime that meets the description in the previous paragraph. We set $h=0$, and $z=m$ in Lemma \ref{rel} and Lemma \ref{par} to get parameters for this graph. Thus $\Gamma_{o+}(V)$ is an SRG with parameters $v={m \brack 1}_{q}(q^{m-1}+1)$, $k = q{m-1 \brack 1}_{q}(q^{m-2}+1)$, $\lambda=q{m-1 \brack 1}_{q}(q^{m-2}+1)-1-q^{2m-3}$, and $\mu= {m-1 \brack 1}_{q}(q^{m-2}+1)$.

So the eigenvalues of the adjacency matrix $A$ are $(k,r,s)=(k,q^{m-1}-1,-(1+q^{m-2}))$,
with multiplicities $(1,f,g)=\left(1,\ \frac{q(q^{m}-1)(q^{m-2}+1)}{(q^{2}-1)},\ \frac{q^{2}(q^{m-1}+1)(q^{m-1}-1)}{(q^{2}-1)}\right)$.

So $L$ has eigenvalues $(0,t,u)=(0,k-r,k-s)=\left(0,{m-1 \brack 1}_{q}(1+q^{m-1}), {m \brack 1}_{q}(1+q^{m-2})\right)$
with multiplicities $(1,f,g)$. 
\subsection{Submodule Structure.}
We now recall from \S\ref{act} the definitions of $C$, $C'$ $U$, $U'$, $\ov{V}_{r}$ and $\ov{V}_{s}$ in the context of the graph $\Gamma_{o+}$. In this case $\G=\mathrm{O^+}(2m,q)$.
By Corollary 2.10, 6.5, Lemma 6.6 of \cite{ST} we have the following result.
\begin{theorem}\label{sto+}
Let $\lp$ be a prime with $\lp \mid q+1$.
Then the relative positions of $U^{\perp}$, $C$, $\ov{V}_{s}$, $\ov{V}_{r}$, $U$, $U'$, and $\left\langle \mathbf{1} \right\rangle$ in the $\bb{F}_{\lp}\mathrm{O^+}(2m,q)$ submodule structure of $(U')^{\perp}$ are given in the following diagrams. We have $\dim((U')^{\perp})=g+2$.

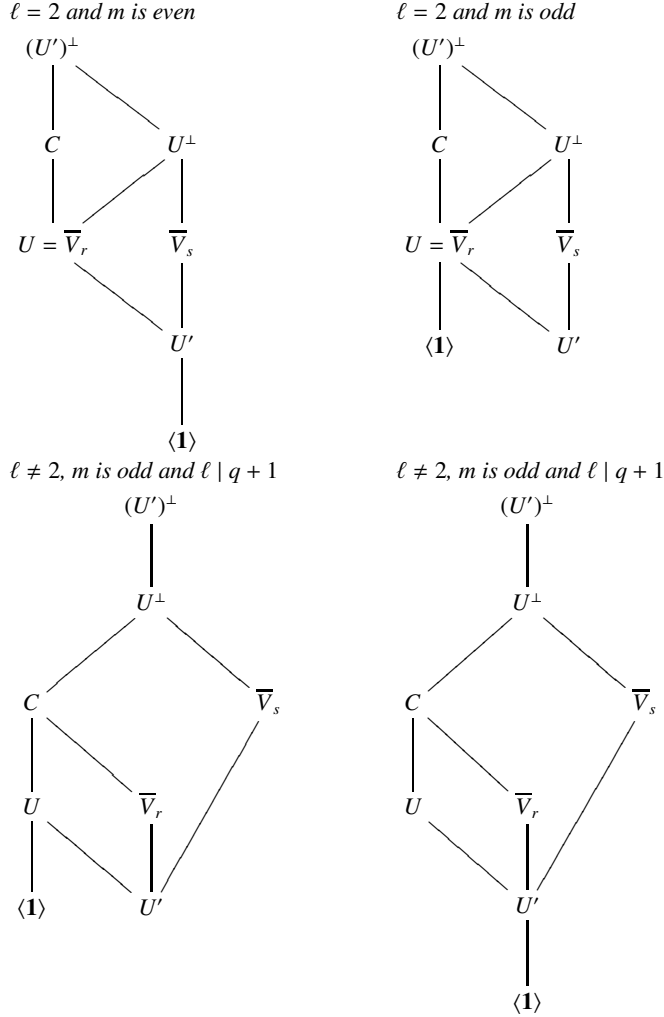
\small{
\begin{minipage}[t][][c]{0.40\textwidth}
$\lp=2$ and $m$ is even

\xymatrix{
 (U')^{\perp} \ar@{-}[d]  \ar@{-}[dr]& \\
C \ar@{-}[d]& U^{\perp} \ar@{-}[dl]  \ar@{-}[d]\\
U=\ov{V}_{r} \ar@{-}[dr] &\ov{V}_{s}\ar@{-}[d]&\\
&U' \ar@{-}[d]  \\
& \left\langle \mathbf{1} \right\rangle  
 }
\end{minipage}
\begin{minipage}[t][][c]{0.40\textwidth}
$\lp= 2$ and $m$ is odd 

\xymatrix{
 (U')^{\perp} \ar@{-}[d]  \ar@{-}[dr] &&\\
C \ar@{-}[d]& U^{\perp} \ar@{-}[dl]  \ar@{-}[d]&\\
 U=\ov{V}_{r} \ar@{-}[d] \ar@{-}[dr] & \ov{V}_{s} \ar@{-}[d]&\\
\left\langle \mathbf{1} \right\rangle &U' & 
}
\end{minipage}

\begin{minipage}[t][][c]{0.40\textwidth}
$\lp\neq 2$, $m$ is odd and $\lp \mid q+1$

\xymatrix{
& (U')^{\perp} \ar@{-}[d] &\\
& U^{\perp} \ar@{-}[dl]  \ar@{-}[dr] & \\
C \ar@{-}[d] \ar@{-}[dr] && \ov{V}_{s} \ar@{-}[ddl]\\
U \ar@{-}[d] \ar@{-}[dr] & \ov{V}_{r} \ar@{-}[d] & \\
 \left\langle \mathbf{1} \right\rangle & U' &
 }
\end{minipage}
\begin{minipage}[t][][c]{0.40\textwidth}
$\lp\neq 2$, $m$ is odd and $\lp \mid q+1$

\xymatrix{
& (U')^{\perp} \ar@{-}[d] &\\
& U^{\perp} \ar@{-}[dl]  \ar@{-}[dr] & \\
C \ar@{-}[d] \ar@{-}[dr] && \ov{V}_{s} \ar@{-}[ddl]\\
U  \ar@{-}[dr] & \ov{V}_{r} \ar@{-}[d] & \\
  &U' \ar@{-}[d] &\\
  &\left\langle \mathbf{1} \right\rangle
 }
\end{minipage}
}

\normalsize Here, $\dim(U')=f-1$, $\dim(U)=f$, $\dim(C)=f+1$, $\dim(U^{\perp})=f+g+1-\dim(U)=g+1$.
\end{theorem}
\subsection{$2$-elementary divisors of $S$ and $K$ when $m$ is even.}
\paragraph{Elementary divisors of $S$.}
We identify $\ov{M}_{i}(A)$ with $\ov{M}_{i}$ and $\ov{A}_{2}$ with $\ov{A}$.
In this case $m$ is even. Therefore ${m-1 \brack 1}_{q}$ is odd. So we have $v_{2}(r)=v_{2}(q^{m-1}-1)=v_{2}(q-1)$, $v_{2}(s)=v_{2}(q^{m-2}+1)=1$, and $v_{2}(k)=v_{2}(\mu)=1$. Setting $v_{2}(q-1)=w$, we have $v_{2}(|S|)=wf+g+1$.

Lemma \ref{eigenval} implies  $\ov{V}_{r} \subset \ov{M}_{1}$ and  $\ov{V}_{s} \subset \ov{M}_{1}$. Thus by Theorem \ref{sto+}, we see that $\ov{M}_{1} \supset U^{\perp}$, and hence $\dim(\ov{M}_{1}) \geq g+1$.
 Lemma \ref{subair1} gives us $U \subset \ov{M}_{w+1}$. Thus by  Theorem \ref{sto+},  we get $\dim(\ov{M}_{w+1}) \geq f$.
 
We have $1(g+1-f)+(w+1)f=v_{2}(|S|)$.
So by Lemma \ref{main}, setting $j=2$, $s_{1}=g+1$, $s_{2}=f$, $s_{3}=\dim(\ov{\ker(A)})=0$, $t_{1}=1$, and $t_{2}=w+1$, we conclude $e_{0}=f$, $e_{1}=g+1-f$, $e_{w+1}=f$, and $e_{i}=0$ for all $i$.
 
\paragraph{Elementary divisors of $K$.}
We identify $\ov{M}_{i}(L)$ with $\ov{M}_{i}$ and $\ov{L}_{2}$ with $\ov{L}$.
In this case $m$ is even, so we have
$v_{2}(s)=v_{2}(q^{m-2}+1)=1$ and $v_{2}(k)=v_{2}(\mu)=1$. Set $v_{2}(q^{m-1}+1)=c$ and $v_{2}({m \brack 1}_{q})=b$. We now have $v_{2}(t)=c$, $v_{2}(u)=b+1$, $v_{2}(v)=c+b$, and $v_{2}(|K|)=cf+(b+1)g-b-c$.

Lemma \ref{subair1} gives us $U'\subset \ov{M_{b+c+1}}$, $U \subset \ov{M_{c+1}}$, and $\ov{V}_{s} \subset \ov{M}_{b+1}$.
We use Lemma \ref{main} to conclude the following.
\begin{enumerate}
\item If $c < b$,  we have $\ov{M}_{c+1} \supset \ov{M}_{b+1}$ and thus $\ov{V}_{s} \subset \ov{M_{c+1}}$. Also since $U \subset \ov{M_{c+1}}$, Theorem \ref{sto+} implies  $U^{\perp} \subset \ov{M}_{c+1}$. Hence $\dim(\ov{M}_{c+1})\geq g+1$.
Again by Theorem \ref{sto+} $\dim(\ov{M}_{b+1}) \geq \dim(\ov{V}_{s}) \geq g$, and $\dim(\ov{M}_{b+c+1}) \geq \dim(U') \geq f-1$. 

So $(c+1)(g+1-g)+(b+1)(g-(f-1))+(b+c+1)(f-1-1)=v_{2}(|K|)$. 
Now by Lemma \ref{main}, setting $j=3$, $s_{1}=g+1$, $s_{2}=g$, $s_{3}=f-1$, $s_{4}=\dim(\ov{\ker(L)})=1$, $t_{1}=c+1$, $t_{2}=b+1$, and $t_{3}=b+c+1$, we have $e_{0}=f$, $e_{c+1}=1$, $e_{b+1}=g-f+1$, $e_{b+c+1}=f-2$, and $e_{i}=0$ for all other $i$.
\item If $c>b$, By arguments similar to those above we can show $\dim(\ov{M}_{b+1})\geq g+1$. We also have $\dim(\ov{M}_{c+1}) \geq f$, and $\dim(\ov{M}_{b+c+1})\geq f-1$.

So
$(b+1)(g+1-f)+(c+1)(f-(f-1))+(b+c+1)(f-1-1)=v_{2}(|K|)$. Applying Lemma \ref{main} as above, we get  $e_{0}=f$, $e_{c+1}=1$, $e_{b+1}=g-f+1$, $e_{b+c+1}=f-2$, and $e_{i}=0$ for all other $i\neq 0$.
\item If $b=c$, by similar arguments, we can show  $e_{0}=f$, $e_{c+1}=g-f+2$, $e_{2c+1}=f-2$, and $e_{i}=0$ for all other $i$.  
\end{enumerate} 
\subsection{$2$-elementary divisors of $S$ and $K$ when $m$ is odd.}
\paragraph{Elementary divisors of $S$.}
We identify $\ov{M}_{i}(A)$ with $\ov{M}_{i}$ and $\ov{A}_{2}$ with $\ov{A}$.
As $m$ is odd, we have $v_{2}(q^{m-1}+1)=1$ and $v_{2}({m-1 \brack 1}_{q})>1$. 
We set $v_{2}({m-1 \brack 1}_{q})=a$, $v_{2}(q-1)=w$ and $v_{2}(q^{m-2}+1)=d$. So $v_{2}(r)=a+w$, $v_{2}(s)=d$, $v_{2}(k)=v_{2}(\mu)=a+d$, and $v_{2}(|S|)=(a+w)f+dg+a+d$.

By Lemma \ref{subair1}, we have $C \subset \ov{M}_{a}$, $U \subset \ov{M}_{a+w+d}$, and 
$\ov{V}_{s} \subset \ov{M}_{d}$. Using Lemma \ref{main} we arrive at the following conclusions.
\begin{enumerate}
\item Assume that $d<a$, then $\ov{M}_{d} \supset \ov{M}_{a} \supset C$. Now since $\ov{V}_{s}$ and $C$ are subsets of $\ov{M}_{d}$,Theorem \ref{sto+} gives us $(U)'^{\perp} \subset \ov{M}_{d}$, and thus $\dim(\ov{M}_{d}) \geq g+2$. Since $C \subset \ov{M_{a}}$, we have $\dim(\ov{M}_{a}) \geq f+1$. Again by Theorem \ref{sto+} we get $\dim(\ov{M}_{a+w+d})\geq f=\dim(U)$. 

Now,
$d(g+2-(f+1))+a(f+1-f)+(a+w+d)(f)=v_{2}(|S|)$.
So by Lemma \ref{main}, setting $j=3$, $s_{1}=g+2$, $s_{2}=f+1$, $s_{3}=f$, $s_{4}=\dim(\ov{\ker(A)})$, $t_{1}=d$, $t_{2}=a$, and $t_{3}=a+w+d$, we have 
$e_{0}=f-1$, $e_{d}=g-f+1$, $e_{a}=1$, $e_{a+w+d}=f$, and $e_{i}=0$ for all other $i$.

\item If $a<d$, by arguments similar to the ones above, we can show $\dim(\ov{M}_{a})\geq g+2$. As $U \subset \ov{M}_{a+w+d} \subset \ov{M}_{d}$ and $\ov{V}_{s} \subset \ov{M}_{d}$, Theorem \ref{sto+} implies $\dim(\ov{M_{d}}) \geq g+1$, and $\dim(\ov{M}_{a+w+d}) \geq f$.

Now,
$a(g+2-(g+1))+d(g+1-f)+(a+w+d)(f)=v_{2}(|S|)$.
Applying Lemma \ref{main} as above, we have 
$e_{0}=f-1$, $e_{d}=g-f+1$, $e_{a}=1$, $e_{a+w+d}=f$, and $e_{i}=0$ for all other $i$. 
\item If $a=d$, by similar arguments we can show that $e_{0}=f-1$, $e_{a}=g+2-f$, $e_{a+d+w}=f$, and $e_{i}=0$ for all other $i$.
\end{enumerate}   
\paragraph{Elementary divisors of $K$.}
We identify $\ov{M}_{i}(L)$ with $\ov{M}_{i}$ and $\ov{L}_{2}$ with $\ov{L}$.
As $m$ is odd, we have $v_{2}(q^{m-1}+1)=1$ and $v_{2}({m-1 \brack 1}_{q})>1$. Set $v_{2}({m-1 \brack 1}_{q})=a$ and $v_{2}(s)=v_{2}(q^{m-2}+1)=d$.  So $v_{2}(t)=a+1$, $v_{2}(u)=d$, $v_{2}(v)=1$, and $v_{2}(|K|)=(a+1)f+dg-1$.

By Lemma \ref{subair1}, we have $C \subset \ov{M}_{a}$ and $U \subset \ov{M}_{a+d+1}$.
Again by Lemma \ref{subair1} we have $\ov{V}_{s}\subset \ov{M}_{d}$.
Using Lemma \ref{main} we arrive at the following conclusions.
\begin{enumerate}
\item Assume $d<a$, then $C\subset \ov{M}_{a} \subset \ov{M}_{d}$. As $\ov{V}_{s} \subset \ov{M}_{d}$, Theorem \ref{sto+} implies $(U')^{\perp} \subset \ov{M}_{d}$, and hence $\dim(\ov{M}_{d}) \geq \dim(U'^{\perp})=g+2$. Also $\dim(\ov{M}_{a})\geq \dim(C)=f+1$, and $\dim(\ov{M}_{a+d+1})\geq \dim(U)=f$.

Now,
$d(g+2-(f+1))+a(f+1-f)+(a+d+1)(f-1)=dg+(a+1)f-a-1=v_{2}(|K|)$.
So by Lemma \ref{main}, setting $j=3$, $s_{1}=g+2$, $s_{2}=g+1$, $s_{3}=f$, $s_{4}=\dim(\ov{\ker(L)})=1$, $t_{1}=a$, $t_{2}=d$, and $t_{3}=a+d+1$, we have
$e_{0}=f-1$, $e_{a}=1$, $e_{d}=g+1-f$, $e_{a+d+1}=f-1$, and $e_{i}=0$ for all $i$.
\item If $a<d$, by arguments similar to those above, $\dim(\ov{M}_{a}) \geq g+2$. As $\mathbf{1} \in \ov{M}_{d}$, and $\ov{V}_{s} \subset \ov{M}_{d}$, by Theorem \ref{sto+}, it follows that $\dim(\ov{M}_{d}) \geq g+1$. We also have $\dim(\ov{M}_{a+d+1}) \geq f$.

Now,
$a(g+2-(g+1))+d(g+1-f)+(a+d+1)(f-1)=dg+(a+1)f-a-1=v_{2}(|K|)$.
By applying Lemma \ref{main} as above, we have $e_{0}=f-1$, $e_{a}=1$, $e_{d}=g+1-f$, $e_{a+d+1}=f-1$, and $e_{i}=0$ for all other $i$.
\item If $a=d$, by arguments similar to those above we may show that $e_{0}=f-1$, $e_{a}=g+2-f$, $e_{2a+1}=f-1$, and $e_{i}=0$ for all other $i$.  
\end{enumerate}

\subsection{$\lp$-elementary divisors of $S$ and $K$ when $m$ is odd, and $\lp \mid q+1$.}
\paragraph{Elementary divisors of $S$.}
We identify $\ov{M}_{i}(A)$ with $\ov{M}_{i}$ and $\ov{A}_{\lp}$ with $\ov{A}$.
In this case $\lp\mid q+1$ is an odd prime with $m$ odd. So we have $v_{\lp}({m-1 \brack 1}_{q})=v_{\lp}(r)$. We set
$v_{\lp}({m-1 \brack 1}_{q})=v_{\lp}(r)=a$, and $v_{\lp}(s)=v_{\lp}(q^{m-2}+1)=d$. Then $v_{\lp}(k)=a+d=v_{\lp}(\mu)$ and $v_{\lp}(|S|)=af+dg+a+d$.

By Lemma \ref{subair1}, we have $U \subset \ov{M}_{a+d}$, $\ov{V}_{s} \subset \ov{M}_{d}$, and $\ov{V}_{r} \subset \ov{M}_{a}$. We now apply Theorem \ref{sto+} and Lemma \ref{main} to conclude the following.

\begin{enumerate}
\item Assume $a<d$, then  $\ov{V}_{s}\subset \ov{M}_{a} \subset \ov{M}_{d}$. Since 
$\ov{V}_{r} \subset \ov{M}_{a}$, Theorem \ref{sto+} implies $\dim(\ov{M}_{a})\geq \dim(U^{\perp})=g+1$. As  $U \subset \ov{M}_{a+d} \subset \ov{M}_{a}$, Theorem \ref{sto+} implies $\dim(\ov{M}_{a}) \geq \dim(U^{\perp}) \geq g+1$. From above, we have $\dim(\ov{M}_{a+d}) \geq \dim(U)=f$.

Now,
$a(g+2-(g+1))+d(g+1-f)+(a+d)f=v_{\lp}(|S|)$.
By Lemma \ref{main}, setting $j=3$, $s_{1}=g+2$, $s_{2}=g+1$, $s_{3}=f$, $s_{4}=\dim(\ov{\ker(A)})=1$, $t_{1}=a$, $t_{2}=d$, and $t_{3}=a+d$, we have
$e_{0}=f-1$, $e_{a}=1$, $e_{d}=g-f+1$, $e_{a+d}=f$, and $e_{i}=0$ for all other $i$.
\item If $d<a$, by arguments similar to those above, we can show that $\dim(\ov{M}_{d}) \geq g+1$. Since $U \subset \ov{M}_{a+d} \subset \ov{M}_{a}$ and $\ov{V}_{r} \subset \ov{M}_{a}$, Theorem \ref{sto+} implies $\dim(\ov{M}_{a}) \geq \dim(C)\geq f+1$. From above, we have $\dim(\ov{M}_{a+d}) \geq \dim(U)=f$. 

Now,
$b(g+2-(f+1))+d(f+1-f)+(a+d)f=v_{\lp}(|S|)$. 
By applying Lemma \ref{main}, we get $e_{0}=f-1$, $e_{a}=1$, $e_{d}=g-f+1$, $e_{a+d}=f$, and $e_{i}=0$ for all other $i$.
\item If $a=d$, by arguments similar to those above, we get $e_{0}=f-1$,  $e_{a}=g-f+2$, $e_{a+d}=f$, and  $e_{i}=0$ for all other $i$. 
\end{enumerate} 
\paragraph{Elementary divisors of $K$.}
We identify $\ov{M}_{i}(L)$ with $\ov{M}_{i}$ and $\ov{L}_{\lp}$ with $\ov{L}$.
In this case $\lp\mid q+1$ is an odd prime with $m$ odd, we have $v_{\lp}({m-1 \brack 1}_{q})>0$, $v_{\lp}(q^{m-1}+1)=0$, and $v_{\lp}({m \brack 1}_{q})=0$. Setting $v_{\lp}({m-1 \brack 1}_{q})=a$ $v_{2}(s)=v_{\lp}(q^{m-2}+1)=d$, we have $v_{\lp}(t)=a$, $v_{\lp}(u)=c$, $v_{\lp}(v)=0$, and $v_{\lp}(|K|)=af+dg$. 
As $L \equiv -A \pmod {2^{a+d}}$, we have $\ov{M}_{i}(L)=\ov{M}_{i}(A)$ for all $i\leq a+c$. So we have $U \subset \ov{M}_{a+d}$, $\ov{V}_{s} \subset \ov{M}_{d}$, and $\ov{V}_{r} \subset \ov{M}_{a}$

 Using this fact and Lemma \ref{main} we conclude the following. 
\begin{enumerate}
\item If $a<d$, then $\ov{M}_{a}\supset \ov{M}_{d} \supset \ov{V}_{s}$. Since $\ov{M}_{a} \supset \ov{V}_{r}$ as well, by Theorem \ref{sto+} we have $\dim(\ov{M}_{a}) \geq g+2$, $\dim(\ov{M}_{d}) \geq g+1$, and $\dim(\ov{M}_{a+d}) \geq f$. 

Now
$a(g+2-(g+1))+d(g+1-f)+(a+b)(f-1)=v_{\lp}(|K|)$.
So by Lemma \ref{main}, setting $j=3$, $s_{1}=g+2$, $s_{2}=g+1$, $s_{3}=f$, $s_{4}=\dim(\ov{\ker(L)})=1$, $t_{1}=a$, $t_{2}=d$, and $t_{3}=a+d$, we have $e_{0}=f-1$, $e_{a}=1$, $e_{d}=g-f+1$, and $e_{a+d}=f-1$, and $e_{i}=0$ for all other $i$. 
\item If $a>d$, we have $M_{d}\supset M_{a}$. So by similar arguments $\dim(\ov{M}_{d}) \geq g+2$. And by the above we have $\dim(\ov{M}_{a}) \geq g+1$, and $\dim(\ov{M}_{a+d}) \geq f$. 

Now
$a(g+2-(f+1))+d(f+1-f)+(a+d)(f-1)=v_{\lp}(|K|)$. By Lemma \ref{main}, we have $e_{0}=f-1$, $e_{a}=1$, $e_{d}=g-f+1$, and $e_{a+d}=f-1$, and $e_{i}=0$ for all other $i$. 
\item If $a=d$, by arguments similar to those above, we have $e_{0}=f-1$, $e_{a}=g-f+2$, $e_{2a}=f-1$, and $e_{i}=0$ for all other $i$.   
\end{enumerate}
\subsection{$\lp$-elementary divisors of $K$ when $m$ is even and $\lp \mid q+1$.\\}

In this case we have $q \equiv -1 \pmod{\lp}$ and thus  $r \equiv s \equiv -2 \pmod{\lp}$ and thus $\lp \nmid |S|$. However we see that $\lp\mid t$ and $\lp \mid u$ and thus $\lp \mid |K|$. In this section we compute the $\lp$-elementary divisors of $K$.
\paragraph{Elementary divisors of $K$.}
We identify $\ov{M}_{i}(L)$ with $\ov{M}_{i}$ and $\ov{L}_{\lp}$ with $\ov{L}$.
Set $v_{\lp}(q^{m-1}+1)=c$ and $v_{\lp}({m \brack 1}_{q})=b$. Since $m$ is odd, we have $v_{\lp}(s)=v_{\lp}(q^{m-2}+1)=0$. So $v_{\lp}(t)=c$, $v_{\lp}(u)=b$, $v_{\lp}(v)=c+b$, $v_{\lp}(\mu)=0$, and $v_{\lp}(|K|)=cf+bg-(c+b)$.

Lemma \ref{subair1} gives us
$\ov{V}_{r} \subset \ov{M_{c}}$, $\ov{V}_{s} \subset \ov{M}_{b}$ and $U' \subset \ov{M}_{b+c}$. We now use Lemma \ref{main} to conclude the following.
\begin{enumerate}
\item If $c < b$,  we have $\ov{M}_{c} \supset \ov{M}_{b}$ and thus $\ov{V}_{s} \subset \ov{M_{c}}$. Also since $\ov{V}_{r} \subset \ov{M_{c}}$, Theorem \ref{sto+} implies  $U^{\perp} \subset \ov{M}_{c}$. Hence $\dim(\ov{M}_{c})\geq g+1$.
Again by Theorem \ref{sto+} $\dim(\ov{M}_{b}) \geq \dim(\ov{V}_{s}) \geq g$, and $\dim(\ov{M}_{b+c}) \geq \dim(U') \geq f-1$. 

So $(c)(g+1-g)+(b)(g-(f-1))+(b+c)(f-1-1)=v_{2}(|K|)$. 
Now by Lemma \ref{main}, setting $j=3$, $s_{1}=g+1$, $s_{2}=g$, $s_{3}=f-1$, $s_{4}=\dim(\ov{\ker(L)})=1$, $t_{1}=c$, $t_{2}=b$, and $t_{3}=b+c$, we have $e_{0}=f$, $e_{c}=1$, $e_{b}=g-f+1$, $e_{b+c}=f-2$, and $e_{i}=0$ for all other $i$.
\item If $c>b$, By arguments similar to those above we can show $\dim(\ov{M}_{b})\geq g+1$. We also have $\dim(\ov{M}_{c}) \geq f$, and $\dim(\ov{M}_{b+c})\geq f-1$.

So
$(b)(g+1-f)+(c)(f-(f-1))+(b+c)(f-1-1)=v_{2}(|K|)$. Applying Lemma \ref{main} as above, we get  $e_{0}=f$, $e_{c}=1$, $e_{b}=g-f+1$, $e_{b+c}=f-2$, and $e_{i}=0$ for all other $i\neq 0$.
\item If $b=c$, by similar arguments, we can show  $e_{0}=f$, $e_{c}=g-f+2$, $e_{2c}=f-2$, and $e_{i}=0$ for all other $i$.  
\end{enumerate} 

\section{$\lp$-elementary divisors of $S$ and $K$ when $\Gamma(V)=\Gamma_{ue}(q,m)$, and $\lp \mid q+1$.}
\label{ue}
Given the graph $\Gamma_{ue}(q,m)$ and a prime $\lp$, table \ref{nilconst} shows that $\ov{A}_{\lp}$ (equivalently $\ov{L}_{\lp}$) is nilpotent if and only if $\lp \mid q+1$. In this section we compute the $\lp$-elementary divisors of $S$ and $K$ when $\Gamma(V)= \Gamma_{ue}(q,m)$ and $\lp \mid q+1$. 

 We set $h=\frac{1}{2}$, and $z=m$ in Lemma \ref{rel} and Lemma \ref{par} to get parameters for this graph. Thus $\Gamma_{ue}(q,m)$ is an SRG with parameters 
$v={m \brack 1}_{q^{2}}(q^{2m-1}+1)$, $k= q^{2}{m-1 \brack 1}_{q^{2}}(q^{2m-3}+1)$ ,$\ \lambda =
(q^{2}-1)+q^{4}((q)^{2m-5}+1){m-2 \brack 1}_{q^{2}}$ , and $\mu ={m-1 \brack 1}_{q^{2}}(q^{2m-3}+1)$.

The adjacency matrix $A$ has eigenvalues $(k,r,s)=(k,q^{2m-2}-1,-(1+q^{2m-3}))$,
with multiplicities $(1,f,g)=\left(1, \dfrac{q^{2}{m \brack 1}_{q^{2}}(q^{2m-3}+1)}{q+1}, \dfrac{q^{3}{m-1 \brack 1}_{q^{2}}(q^{2m-1}+1)}{q-1} \right)$.
So $L$ has eigenvalues $(0,t,u)=\left(0,{m-1 \brack 1}_{q^{2}}(1+q^{2m-1}), {m \brack 1}_{q^{2}}(1+q^{2m-3})\right)$
with multiplicities $(1,f,g)$.
\subsection{Submodule Structure}We now recall from \S\ref{act} the definitions of $C$, $C'$ $U$, $U'$, $\ov{V}_{r}$ and $\ov{V}_{s}$ in the context of the graph $\Gamma_{ue}$. In this case $\G=\mathrm{U}(2m,q^2)$.
By Corollary $2.10$, Corollary $4.5$, and Lemma 4.6 of \cite{ST}, we have the following result.
\begin{theorem}\label{stoue}
When $\lp \mid q+1$, the module $U'^{\perp}$ has a submodule $M$ containing $\ov{V}_{s}$ such that $\dim(M/\ov{V}_{s})=1$.
The relative positions of $M$,$U^{\perp}$, $C$, $\ov{V}_{s}$, $\ov{V}_{r}$, $U$, $U'$, and $\left\langle \mathbf{1} \right\rangle$ in the $\bb{F}_{\lp}\mathrm{U}(2m,q^2)$ submodule structure of $(U')^{\perp}$ are given in the following  diagrams.\\
\small{\begin{minipage}[t][][c]{0.48 \textwidth}
$\lp \nmid m$\\
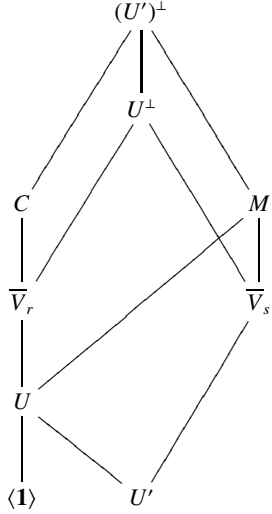
\xymatrix{
& (U')^{\perp} \ar@{-}[d] \ar@{-}[ddl] \ar@{-}[ddr] &\\
& U^{\perp} \ar@{-}[ddl] \ar@{-}[ddr]&\\
C \ar@{-}[d] & & M \ar@{-}[d] \ar@{-}[ddll]  \\
\ov{V}_{r}\ar@{-}[d] & & \ov{V}_{s} \ar@{-}[ddl] \\
 U \ar@{-}[d]  \ar@{-}[dr] & &\\
\left\langle \mathbf{1} \right\rangle & U' & 
}
\end{minipage}
\begin{minipage}[t][][c]{0.45 \textwidth}
$\lp \mid m$ \\
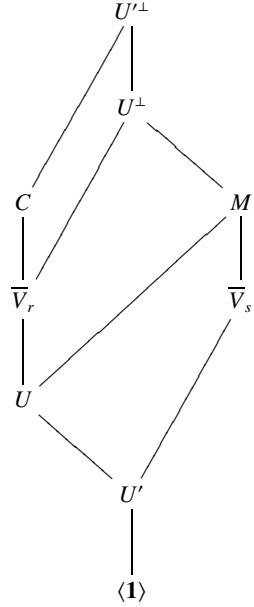
\xymatrix{
& U'^{\perp}  \ar@{-}[ddl] \ar@{-}[d]& \\
& U^{\perp} \ar@{-}[dr]  \ar@{-}[ddl]& \\
C \ar@{-}[d]& & M \ar@{-}[ddll] \ar@{-}[d]\\
\ov{V}_{r}\ar@{-}[d] & &\ov{V}_{s} \ar@{-}[ddl] \\
U \ar@{-}[dr] & & \\
& U' \ar@{-}[d]&\\
& \left\langle \mathbf{1} \right\rangle &
}
\end{minipage}}

\normalsize Here $\delta:=\dim(\ov{V}_{r}/U)=\dfrac{q^{2m}-1}{q+1}$, $y:=\dim(\ov{V}_{s}/U')=\dfrac{(q^{2m}-1)(q^{2m-1}-q)}{(q+1)^{2}}$, $x:=\dim(U')= \dfrac{(q^{2m}-1)(q^{2m-1}+1)}{(q^{2}-1)(q-1)}$, and $\dim(M/\ov{V}_{s})=\dim(U/U')=\dim(C/ \ov{V}_{r})=1$.
\end{theorem}
\subsection{$\lp$-elementary divisors of $S$ and $K$ when $\lp \mid q+1$ and $\lp \nmid m$.}
\paragraph{Elementary divisors of $S$.}
We identify $\ov{M}_{i}(A)$ with $\ov{M}_{i}$ and $\ov{A}_{\lp}$ with $\ov{A}$.
In this case, $\lp \mid q+1$ and $\lp \nmid m$.
Set $v_{\lp}(q^{2}-1)=w$, $v_{\lp}({m-1 \brack 1}_{q^{2}})=a$, and $v_{\lp}(q^{2m-3}+1)=d$. Then $v_{\lp}(r)=w+a$, $v_{\lp}(s)=d$, $v_{\lp}(k)=v_{\lp}(\mu)=a+d$, and $v_{\lp}(|S|)=(w+a)f+dg+a+d$.

Now $s$ is an eigenvalue of valuation $d$ and $k$ is an eigenvalue of valuation $a+d$. 
So Lemma \ref{eigenval} implies $\ov{V}_{s} \subset \ov{M}_{d}$ and $\left\langle \mathbf{1} \right\rangle \subset \ov{M}_{a+d} \subset \ov{M}_{d}$. So by Theorem \ref{stoue}, we have $\ov{M}_{d} \supset \ov{V}_{s}\oplus \left\langle \mathbf{1} \right\rangle$.

Lemma \ref{subair1} implies $C \subset \ov{M}_{a}$, $U \subset \ov{M}_{w+a+d}$, and $\ov{V}_{r} \subset \ov{M_{w+a}}$.

For any positive integer $n$, we have $q^{2n+1}+1=(q+1)\left({2n+1 \brack 1}_{-q}\right)$. Therefore if $\lp \mid q+1$, the following are true.

1) $v_{\lp}(q^{2n+1}+1)=v_{\lp}(q+1)$ if and only if $\lp \mid 2n+1$.

2) $v_{\lp}({m-1 \brack 1}_{q^{2}})=0$ if and only if $\lp \mid m-1$. 

3) $v_{\lp}(q^{2}-1)=v_{\lp}(q+1)$ if and only if $\lp\neq 2$.

\subparagraph*{Subcase 1:When $\lp \nmid m-1$.}

In this case, $a=0$, since $v_{\lp}({m-1 \brack 1}_{q^{2}})=0$. We have $v_{\lp}(|S|)=wf+dg+d$. We apply Lemma \ref{main} and Theorem \ref{stoue} to arrive at the following results.
\begin{enumerate}
\item Assume $w< d $, then we have as $\ov{M}_{d} \supset \ov{V}_{s}\oplus \left\langle \mathbf{1} \right\rangle $, and $\ov{V}_{r} \subset \ov{M}_{d} \subset \ov{M}_{w} \supset \ov{V}_{r}$. So by Theorem \ref{stoue}, $U^{\perp}= \ov{V}_{r}+ \ov{V}_{s} \supset \ov{M}_{w}$. We saw that $\ov{M}_{w+d} \supset U$.
Again by Theorem \ref{stoue}, we have $\dim(\ov{M}_{w})\geq \dim(U^{\perp})=f+g+1-\dim(U)=f+g-x$, $\dim(\ov{M}_{d}) \geq g+1$, and $\dim(\ov{M}_{w+d}) \geq \dim(U)=x+1$.

Now $a(f+g-x-(g+1))+d((g+1)-(x+1))+(w+d)(x+1)=wf+dg+d=v_{\lp}(|S|)$. 
So by Lemma \ref{main}, setting $j=3$, $s_{1}=f+g-x$, $s_{2}=g+1$, $s_{3}=x+1$, $s_{4}=\dim(\ov{\ker(A)})=0$, $t_{1}=w$, $t_{2}=d$, and $t_{3}=w+d$, we have $e_{0}=x+1$, $e_{w}=f-x-1$, $e_{d}=g-x$, $e_{w+d}=x+1$, and $e_{i}=0$ for all other $i$.  
\item If $w >d$, by arguments similar to those above we can show that $\ov{M}_{d} \supset U^{\perp}$, $\ov{M}_{w} \supset \ov{V}_{r}$, and $M_{w+d} \supset U$. Applying Lemma \ref{main} as above, we can conclude that $e_{0}=x+1$, $e_{w}=f-x-1$, $e_{d}=g-x$, $e_{w+d}=x+1$, and $e_{i}=0$ for all other $i$.  
\item If $w=d$, again by arguments similar to those above, we can show that $e_{0}=x+1$, $e_{w}=f+g-2x-1$, $e_{w+d}=x+1$, and $e_{i}=0$ for all other $i$.   
\end{enumerate}
\subparagraph*{Subcase 2: When $\lp \mid m-1$.}
In this case, $a \neq 0$, but $\lp \nmid 2m-3$. So $d=v_{\lp}(q^{2m-3}+1) =v_{\lp}(q+1)\leq v_{\lp}(q^{2}-1)=w$, with the equality holding if and only if $\lp \neq 2$.
So we have either $a \leq d < w+a < w+a+d$, or $d<a<w+a<w+a+d$.
\begin{enumerate}
\item If $a < d < w+a<w+a+d$, we have $\ov{M}_{d}\subset \ov{M}_{a}$. So by Theorem \ref{stoue} $\ov{M}_{a}\supset \ov{M}_{d}\supset C+\ov{V}_{s}=(U')^{\perp}$. Since $\ov{M}_{w+a} \supset \ov{V}_{r}$ and $d <w+a$, Theorem \ref{stoue} implies $\ov{M}_{d} \supset \ov{V}_{s}+\ov{V}_{r}=U^{\perp}$. We also have $\ov{M}_{w+a} \supset \ov{V}_{r}$ and $\ov{M}_{w+a+d}\supset U$. Thus we have $\dim(\ov{M}_{a})\geq \dim(U'^{\perp})=f+g+1-x$, $\dim(\ov{M}_{w})\geq \dim(U^{\perp})=f+g-x$, $\dim(\ov{M}_{w+a})\geq  f$, and $\dim(\ov{M}_{w+a+d}) \geq x+1$. 

Now,
\begin{align*}
a(f+g+1-x-(f+g-x))+d(f+g-x-f)+(w+d)(f-x-1)\\
+(w+a+d)(x+1) \\
= (w+a)f+dg  + d+a
=v_{\lp}(|S|).
\end{align*}
Thus by Lemma \ref{main}, setting $j=4$, $s_{1}=f+g+1-x$, $s_{2}=f+g-x$, $s_{3}=f$, $s_{4}=x+1$, $s_{5}=\dim(\ov{\ker(A)})=0$, $t_{1}=a$, $t_{2}=d$, $t_{3}=w+a$, and $t_{4}=w+a+d$, we conclude that $e_{0}=x$, $e_{a}=1$, $e_{d}=g-x$, $e_{w+a}=f-x-1$, $e_{w+a+d}=x+1$, and $e_{i}=0$ for all other $i$.
\item If $d<a<w+a<w+a+d$, by arguments similar to those above, we can show 
$\ov{M}_{d} \supset (U')^{\perp}$,  $\ov{M}_{a} \supset C$ $\ov{M}_{w+d} \supset \ov{V}_{r}$, and $\ov{M}_{w+a+d} \supset U$. Now by applying Lemma \ref{main} like above, we have $e_{0}=x$, $e_{a}=1$, $e_{d}=g-x$, $e_{w+a}=f-x-1$, $e_{w+a+d}=x+1$, and $e_{i}=0$ for all other $i$.
\item If $d=a<w+a<w+a+d$, by similar arguments, 
$\ov{M}_{a} \supset (U')^{\perp}$, $\ov{M}_{w+a} \supset \ov{V}_{r}$, and $\ov{M}_{w+a+d} \supset U$.
Now by applying Lemma \ref{main} like above, we can show that $e_{0}=x$, $e_{a}=e_{d}=g+1-x$, $e_{w+a}=f-x-1$, $e_{w+a+d}=x+1$, and $e_{i}=0$ for all other $i$.
\end{enumerate}

\paragraph{Elementary divisors of $K$.}
We identify $\ov{M}_{i}(L)$ with $\ov{M}_{i}$ and $\ov{L}_{\lp}$ with $\ov{L}$.

We set $v_{\lp}({m-1 \brack 1}_{q^{2}})=a$, $v_{\lp}(q^{2m-3}+1)=d$, and $v_{\lp}(q^{2m-1}+1)=c$. Then $v_{\lp}(k)=v_{\lp}(\mu)=a+d$, $v_{\lp}(v)=v_{\lp} \left({m \brack 1}_{q^{2}} (q^{2m-1}+1)\right)=c$. So $v_{\lp}(t)=a+c$, $v_{\lp}(u)=d$, and $v_{\lp}(|K|)=(a+c)f+dg-c$.

As $L(\mathbf{1})=0$, we have $\left\langle \mathbf{1} \right\rangle \subset \ov{M}_{i}$ for all $i$. 
Since $s$ is an eigenvalue of valuation $d$, Lemma \ref{eigenval} implies $\ov{M}_{d} \supset \ov{V}_{r}$. So by Theorem \ref{stoue}, we see that  $\ov{M}_{d} \supset \ov{V}_{s} \oplus \left\langle \mathbf{1} \right\rangle$.
Since $t$ is an eigenvalue of valuation $a+c$, Lemma \ref{eigenval} implies $\ov{V}_{r} \subset \ov{M}_{a+c}$.
Lemma \ref{subair1} gives us $C \subset \ov{M}_{a}$ and $U' \subset \ov{M}_{a+c+d}$. Thus by Theorem \ref{stoue}, $U'\oplus \left\langle \mathbf{1} \right\rangle = U\subset \ov{M}_{a+c+d}$.
\subparagraph*{Subcase 1: When $\lp \nmid m-1$.}
In this case, $a=0$, as $v_{\lp}({m-1 \brack 1}_{q^{2}})=0$. Thus $v_{\lp}(|K|)=cf+dg-c$. We apply Lemma \ref{eigenval}, and Theorem \ref{stoue} to conclude the following.
\begin{enumerate}
\item If $c<d<d+c$,
From the information we gathered above, we have 
$\ov{M}_{c} \supset \ov{V}_{r}$ and $\ov{M}_{c}\supset \ov{M}_{d} \supset \ov{V}_{s} \oplus \left\langle \mathbf{1} \right\rangle$. Applying Theorem \ref{stoue} gives us $\ov{M}_{c} \supset U^{\perp}$ and hence $\dim(\ov{M}_{c}) \geq f+g-x$.
We also have by Theorem \ref{stoue}, $\dim(\ov{M}_{d}) \geq g+1$ and $\dim(\ov{M}_{d+c}) \geq x+1$.

Now we have, $c(f+g-x-(g+1))+d(g+1-(x+1))+(c+d)(x+1-1)=cf+dg-c=v_{\lp}(|K|)$.
So by Lemma \ref{main}, setting $j=3$, $s_{1}=f+g-x$, $s_{2}=g+1$, $s_{3}=x+1$, $s_{4}=\dim(\ov{\ker(L)})=1$, $t_{1}=c$, $t_{2}=d$, and $t_{3}=c+d$, we have $e_{0}=x+1$, $e_{c}=f-x-1$, $e_{d}=g-x$, $e_{c+d}=x$, and $e_{i}=0$ for all other $i$.
\item If $d<c<d+c$,
By arguments similar to those above, we can show 
$\ov{M}_{d} \supset U^{\perp}$, and $\dim(\ov{M}_{d}) \geq f+g-x$;
$\ov{M}_{c} \supset \ov{V}_{r}$, and $\dim(\ov{M}_{c}) \geq f+1$; and
$\ov{M}_{d+c} \supset U$, and $\dim(\ov{M}_{d+c}) \geq x+1$.
By applying Lemma \ref{main} as above, we can show that $e_{0}=x+1$, $e_{c}=f-x-1$, $e_{d}=g-x$, $e_{c+d}=x$, and $e_{i}=0$ for all other non-zero $i$.
\item If $c=d<d+c$, by arguments similar to those above we can show $e_{0}=x+1$, $e_{c}=f+g-2x-1$, $e_{c+d}=x$, and $e_{i}=0$ for all other $i$.  
\end{enumerate}
\subparagraph*{Subcase 2: When $\lp \mid m-1$.}
 As $q \equiv -1 \pmod \lp$ and $\lp \mid m-1$, by the observations at the beginning of the subsection, we have $c=d$. We apply Lemma \ref{main} and Theorem \ref{stoue} to conclude the following.
  
\begin{enumerate}
\item Assume that $a< c=d < a+d < a+c+d$. As $C \subset \ov{M}_{a}$, and $\ov{V}_{s} \subset \ov{M}_{c} \subset \ov{M}_{d}$, by Theorem \ref{stoue} $\ov{M}_{d} \supset (U')^{\perp}$, and thus $\dim(\ov{M}_{d})\geq f+g+1-x$. Also by Theorem \ref{stoue}, since $\ov{V}_{r}\subset \ov{M}_{a+d} \subset \ov{M}_{c}$,  $\ov{M}_{c} \supset U^{\perp}$ and thus $\dim(\ov{M}_{c})\geq f+g-x$. We also have $U \subset \ov{M}_{a+c+d}, $ $\ov{V}_{r} \subset \ov{M}_{a+d}$, and thus $\dim(\ov{M}_{d+a}) \geq f$, and $\dim(\ov{M}_{a+c+d})\geq x+1$.

We have
$a(f+g+1-x-(f+g-x))+c(f+g-x-(f))+(a+d)(f-(x+1))+(a+c+d)(x+1-1)=(a+d)f+cg-(a+d)=v_{\lp}(|K|)$.
So by Lemma \ref{main}, setting $j=4$, $s_{1}=f+g+1-x$, $s_{2}=f+g-x$, $s_{3}=f$, $s_{4}=x+1$ $s_{5}=\dim(\ov{\ker(L)})=1$, $t_{1}=a$, $t_{2}=c$, $t_{3}=a+d$, and $t_{4}=a+c+d$,  we have $e_{0}=x$, $e_{a}=1$, $e_{c}=g-x$, $e_{a+d}=f-x-1$, $e_{a+c+d}=x$, and $e_{i}=0$ for all other non-zero $i$.  
\item Assume that $c=d<a < a+d < b+c+d$. Then by arguments similar to those above, $\ov{M}_{c} \supset C+\ov{V}_{s}=(U')^{\perp}$, $\ov{M}_{a} \supset C$, $\ov{M}_{a+d} \supset \ov{V}_{r}$, and $\ov{M}_{a+c+d} \supset U$. Now applying Lemma \ref{main} as above, we have $e_{0}=x$, $e_{a}=1$, $e_{c}=g-x$, $e_{a+d}=f-x-1$, $e_{a+c+d}=x$, and $e_{i}=0$ for all other non-zero $i$.  
\item If $c=d=a < a+d < a+c+d$, then by arguments similar to those above, we can show $e_{0}=x$, $e_{c}=g+1-x$, $e_{a+d}=f-x-1$, $e_{a+c+d}=x$, and $e_{i}=0$ for all other $i$.      
\end{enumerate}

\subsection{$\lp$-elementary divisors of $S$ and $K$ when $\lp \mid q+1$, and $\lp \mid m$.} 
\paragraph{Elementary divisors of $S$.}
We identify $\ov{M}_{i}(A)$ with $\ov{M}_{i}$ and $\ov{A}_{\lp}$ with $\ov{A}$.
In this case $\lp \nmid m-1$ and thus $\lp \nmid {m-1 \brack 1}_{q^{2}}$ and $v_{\lp}(q^{2m-2}-1)=v_{\lp}(q^{2}-1)$. As $\lp \nmid 2m-3$ and $\lp \nmid 2m-1$, we have $v_{\lp}(s)=v_{\lp}(q^{2m-3}+1)=v_{\lp}(q^{2m-1}+1)=v_{\lp}(q+1)\leq v_{\lp}(q^{2}-1)$.
 Set $w=v_{\lp}(q^{2}-1)=v_{\lp}(r)$ ,$v_{\lp}(q^{2m-3}+1)=v_{\lp}(s)=d$. 
 
We have $v_{\lp}(|S|)=wf+dg+d$, and $v_{\lp}(k)=v_{\lp}(\mu)=d$

Observe that $d \leq w < d+w$.
As $r$ is an eigenvalue of valuation $w \geq d$ and $s,k$ are eigenvalues of valuation $d$, by Lemma \ref{eigenval} and Theorem \ref{stoue}, \\
$\ov{M}_{d} \supset \ov{V}_{r}+\ov{V}_{s}=U^{\perp}$, and $\ov{M}_{w} \supset \ov{V}_{r} $. By Lemma \ref{subair1},
$U \subset \ov{M}_{w+d}$. Thus by Theorem \ref{stoue}, $\dim(\ov{M}_{d})\geq f+g-x$, $\dim(\ov{M}_{w}) \geq f$ and $\dim(\ov{M}_{w+d}) \geq x+1$.

We have $d(f+g-x-f)+w(f-(x+1))+(d+w)(x+1)=v_{\lp}(|S|)$. 
So by Lemma \ref{main}, setting $j=3$, $s_{1}=f+g-x$, $s_{2}=f$, $s_{3}=x+1$, $s_{4}=\dim(\ov{\ker(A)})=0$, $t_{1}=d$, $t_{2}=w$, and $t_{3}=w+d$, we have $e_{0}=x+1$, $e_{d}=g-x$, $e_{w}=f-x-1$, $e_{w+d}=x+1$, and $e_{i}=0$ for all other $i$.
\paragraph{Elementary divisors of $K$.}
We identify $\ov{M}_{i}(L)$ with $\ov{M}_{i}$ and $\ov{L}_{\lp}$ with $\ov{L}$.
Since $\lp \nmid m-1$, we have $\lp \nmid {m-1 \brack 1}_{q^{2}}$ and thus $v_{\lp}(q^{2m-2}-1)=v_{\lp}(q^{2}-1)$.  
As $\lp \nmid 2m-3$ and $\lp \nmid 2m-1$, we have $v_{\lp}(s)=v_{\lp}(q^{2m-3}+1)=v_{\lp}(q^{2m-1}+1)=v_{\lp}(q+1)\leq v_{\lp}(q^{2}-1)$. Set $v_{\lp}(q^{2m-1}+1)=d$ and $v_{\lp}\left({m \brack 1}_{q^{2}}\right)=b$.

We have $v_{\lp}(k)=v_{\lp}(\mu)=d$, $v_{\lp}(v)=b+d$, $v_{\lp}(t)=d$, $v_{\lp}(u)=d+b$, and $v_{\lp}(|K|)=df+(d+b)g-(b+d)$.

As $L(\mathbf{1})=0$, we have $\left\langle \mathbf{1} \right\rangle \subset \ov{M}_{i}$ for all $i$.
Since $t$ is an eigenvalue of valuation $d$, and $u$ is an eigenvalue of valuation $b+d$, by Lemma \ref{eigenval} and Theorem \ref{stoue}, we see that 
$\ov{M}_{d} \supset \ov{V}_{r}+\ov{V}_{s}=U^{\perp}$ and $\ov{M}_{d+b}\supset \ov{V}_{s}$.

By Lemma \ref{subair1}, we have $U \subset \ov{M}_{2d}$ and $U'\subset \ov{M}_{b+2d}$. 

We apply Lemma \ref{main} and Theorem \ref{stoue} we arrive at the following conclusions.

\begin{enumerate}
\item If $b<d$, $2d >b+d$, we have $\ov{M}_{b+d}\supset \ov{M}_{2d}+\ov{V}_{s} \supset U+\ov{V}_{s}$. Thus by Theorem \ref{stoue}, we see that $\ov{M}_{b+d} \supset M$ and thus $\dim(\ov{M}_{b+d}) \geq g+1$. Since $\ov{M}_{d} \supset U^{\perp}$, $\ov{M}_{2d} \supset \ov{U}$ and $\ov{M}_{b+2d}\supset U'$, Theorem \ref{stoue} implies $\dim(M)_{d}\geq f+g-x$ $\dim(\ov{M}_{2d}) \geq x+1$, and $\dim(\ov{M}_{b+2d})\geq x$.

Now, $d(f+g-x-(g+1))+(b+d)(g+1-(x+1))+2d(x+1-x)+(b+2d)(x-1)=v_{\lp}(|K|)$.
So by Lemma \ref{main}, setting $j=4$, $s_{1}=f+g-x$, $s_{2}=g+1$, $s_{3}=x+1$, $s_{4}=x$, $s_{5}=\dim(\ov{\ker(L)})=1$, $t_{1}=d$, $t_{2}=b+d$, $t_{3}=2d$, and $t_{4}=b+2d$, we have $e_{0}=x+1$, $e_{d}=f-x-1$, $e_{b+d}=g-x$, $e_{2d}=1$, $e_{b+2d}=x-1$, and $e_{i}=0$ for all other $i$.
\item If $b> d$, by similar arguments, $\ov{M}_{2d} \supset M$, $\ov{M}_{b+d} \supset \ov{V}_{s}$, $\ov{M}_{d} \supset U^{\perp}$, $\ov{M}_{2d} \supset \ov{U}$ and $\ov{M}_{b+2d}\supset U'$. Applying Lemma \ref{main} like in the above case,  we have $e_{0}=x+1$, $e_{d}=f-x-1$, $e_{b+d}=g-x$, $e_{2d}=1$, $e_{b+2d}=x-1$, and $e_{i}=0$ for all other $i$.
\item If $b=d$, by arguments similar to those above, we have $e_{0}=x+1$, $e_{b}=f-x-1$, $e_{b+d}=g-x+1$, $e_{d+2b}=x-1$, and $e_{i}=0$ for all other $i$.   
\end{enumerate}

\section{$\lp$-elementary divisors of $S$ and $K$ when $\Gamma(V)=\Gamma_{uo}(q,m)$, and $\lp \mid q+1$.}\label{last}
Given the graph $\Gamma_{uo}(q,m)$ and a prime $\lp$, table \ref{nilconst} shows that $\ov{A}_{\lp}$ (equivalently $\ov{L}_{\lp}$) is nilpotent if and only if $\lp \mid q+1$. In this section we compute the $\lp$-elementary divisors of $S$ and $K$ when $\Gamma(V)= \Gamma_{uo}(q,m)$ and $\lp \mid q+1$.

From now on in this section, we denote $\Gamma_{uo}(q,m)$ by $\Gamma_{uo}$, and $\lp$ is a prime that meets the description in the previous paragraph.
 We set $h=\frac{3}{2}$, and $z=m$ in Lemma \ref{rel} and Lemma \ref{par} to get parameters for this graph. Thus $\Gamma_{uo}(q,m)$ is an SRG with parameters 
$v={m \brack 1}_{q^{2}}(q^{2m+1}+1)$, $k= q^{2}{m-1 \brack 1}_{q^{2}}(q^{2m-1}+1)$, $\lambda=
(q^{2}-1)+q^{4}((q)^{2m-3}+1){m-2 \brack 1}_{q^{2}}$, and $\mu={m-1 \brack 1}_{q^{2}}(q^{2m-1}+1)$. 

The adjacency matrix $A$ has eigenvalues $(k,r,s)=(k,q^{2m-2}-1,-(1+q^{2m-1}))$
with multiplicities $(1,f,g)=\left(1, \dfrac{q^{3}{m \brack 1}_{q^{2}}(q^{2m-1}+1)}{q+1}, \dfrac{q^{2}{m-1 \brack 1}_{q^{2}}(q^{2m-2}-1)}{q-1} \right)$. 
So the Laplacian $L$ has eigenvalues $(0,t,u)=(0,k-r,k-s)=\left(0,{m-1 \brack 1}_{q^2}(1+q^{2m+1}), {m \brack 1}_{q^2}(1+q^{2m-1})\right)$
with multiplicities $(1,f,g)$. 
\subsection{Submodule structure}
We now recall from \S\ref{act} the definitions of $C$, $C'$ $U$, $U'$, $\ov{V}_{r}$ and $\ov{V}_{s}$ in the context of the graph $\Gamma_{uo}$. In this case $\G=\mathrm{U}(2m+1,q^2)$.
By Corollary $2.10$, Corollary $5.6$, and Proposition $5.14$ of \cite{ST}, we have the following result.
\begin{theorem}\label{stouo}
If $\lp \mid q+1$, the following are true.
\begin{enumerate}
\item If $\lp \nmid m$, then
$C \supset \ov{V}_{r} \supset U=\left\langle \mathbf{1} \right\rangle \oplus \ov{V}_{s} \supset \ov{V}_{s}=U'$.
\item If $\lp \mid m$, then
$C \supset \ov{V}_{r} \supset U \supset \ov{V}_{s}=U' \supset \left\langle \mathbf{1} \right\rangle$.
\item We have $\dim(C)=f+1$ and $\dim(U)=g+1$.
\end{enumerate}
\end{theorem}  
\subsection{Elementary divisors of $S$ and $K$, when $\lp \nmid m$, and $\lp \mid q+1$}
\paragraph{Elementary divisors of $S$.}
We identify $\ov{M}_{i}(A)$ with $\ov{M}_{i}$ and $\ov{A}_{\lp}$ with $\ov{A}$.
We set $v_{\lp}(q^{2}-1)=w$, $v_{\lp}({m-1 \brack 1}_{q^{2}})=a$ and $v_{\lp}(q^{2m-1}+1)=d$. So we have $v_{\lp}(r)=w+a$, $v_{\lp}(s)=d$, $v_{\lp}(k)=v_{\lp}(\mu)=a+d$, and $v_{\lp}(|S|)=(w+a)f+dg+a+d$.

We have $a<w+a<w+a+d$.
As $r$ is an eigenvalue of valuation $w+a$, by Lemma \ref{eigenval}, we have
$\ov{M}_{a+w} \supset \ov{V}_{r}$.
By Lemma \ref{subair1}, $\ov{M}_{a} \supset C$, and $\ov{M}_{w+a+d} \supset U$. Thus by Theorem \ref{stouo}, $\dim(\ov{M}_{a})\geq f+1$, $\dim(\ov{M}_{w+a}) \geq f$, and $\dim(\ov{M}_{w+a+d})\geq g+1$.

Now, $a(f+1-f)+(w+a)(f-(g+1))+(w+a+d)(g+1)=(w+a)f+dg+a+w=v_{\lp}(|S|)$.
So by Lemma \ref{main}, setting $j=3$, $s_{1}=f+1$, $s_{2}=f$, $s_{3}=g+1$, $s_{4}=\dim(\ov{\ker(A)})=0$, $t_{1}=a$, $t_{2}=w+a$, and $t_{3}=w+a+d$, we have $e_{0}=g$, $e_{a}=1$, $e_{w+a}=f-g-1$, $e_{w+a+d}=g+1$, and $e_{i}=0$ for all other $i$.
 
\paragraph{Elementary divisors of $K$.}
We identify $\ov{M}_{i}(L)$ with $\ov{M}_{i}$ and $\ov{L}_{\lp}$ with $\ov{L}$.
We set  $v_{\lp}({m-1 \brack 1}_{q^{2}})=a$, $v_{\lp}(q^{2m-1}+1)=c$, and $v_{\lp}(s)=v_{\lp}(q^{2m+1}+1)=d$ So we have $v_{\lp}(t)=a+d$, $v_{\lp}(u)=c$, $v_{\lp}(v)=d$, and $v_{\lp}(|K|)=(a+d)f+cg-d$.

By Lemma \ref{subair1} we have $C \subset\ov{M}_{a}$ and $U \subset \ov{M}_{a+d+c}$. As $t$ is an eigenvalue of valuation $a+d$ by Lemma \ref{eigenval}, we have
$\ov{M}_{a+d} \supset \ov{V}_{t}=\ov{V}_{r}$.
By Theorem \ref{stouo}, we have $\dim(\ov{M}_{a})\geq f+1$, $\dim(\ov{M}_{a+d})\geq f$, and $\dim(\ov{M}_{a+c+d})\geq g+1$.

Now $a(f+1-f)+(a+d)(f-(g+1))+(a+c+d)(g+1-1)=v_{\lp}(|K|)$.
So by Lemma \ref{main}, setting $j=3$, $s_{1}=f+1$, $s_{2}=f$, $s_{3}=g+1$, $s_{4}=\dim(\ov{\ker(L)})=1$, $t_{1}=a$, $t_{2}=a+b$, and $t_{3}=a+b+c$, we have $e_{0}=g$, $e_{a}=1$, $e_{d+a}=f-g-1$, $e_{d+a+c}=g$, and $e_{i}=0$ for all other $i$.

\subsection{Elementary divisors of $S$ and $K$, when $\lp \mid m$, and $\lp \mid q+1$}
\paragraph{Elementary divisors of $S$.}
We identify $\ov{M}_{i}(A)$ with $\ov{M}_{i}$ and $\ov{A}_{\lp}$ with $\ov{A}$.
As $\lp \mid m$ and $q \equiv -1 \pmod \lp$, we have 
$v_{\lp}(q^{2}-1)=v_{\lp}(r)$, and $v_{\lp}(s)=v_{\lp}(k)=$. Set $v_{\lp}(q^{2}-1)=w$ and $v_{\lp}(q^{2m-1}+1)=d$. We have $v_{\lp}(|S|)=wf+dg+d$.

As $r$ is an eigenvalue of valuation $w$, Lemma \ref{eigenval} implies $\ov{V}_{r} \subset \ov{M}_{w}$.
By Theorem \ref{subair1}, we have $\ov{M}_{w+d} \supset U$.
By Theorem \ref{stouo}, we have $\dim(\ov{M}_{w}) \geq f$, and $\dim(\ov{M}_{w+d})\geq g+1$.

Now, $w(f-(g+1))+(w+d)(g+1) = v_{\lp}(|S|)$.
So by Lemma \ref{main}, setting $j=2$, $s_{1}=f$, $s_{2}=g+1$, $s_{3}=\dim(\ov{\ker(A)})=0$, $t_{1}=w$, and $t_{2}=w+d$, we have $e_{0}=g+1$, $e_{w}=f-g-1$,  $e_{w+d}=g+1$, and $e_{i}=0$ for all other $i$.

\paragraph{Elementary divisors of $K$.}
We identify $\ov{M}_{i}(L)$ with $\ov{M}_{i}$ and $\ov{L}_{\lp}$ with $\ov{L}$.
As $\lp \mid m$, we have $v_{\lp}(q^{2m+1}+1)=v_{\lp}(q^{2m-1}+1)$. We set $v_{\lp}({m \brack 1}_{q^{2}})=b$ and $v_{\lp}(q^{2m+1}+1)=d$. We have $v_{\lp}(t)=d$, $v_{\lp}(u)=b+d$, $v_{\lp}(v)=b+d$, and $v_{\lp}(|K|)=df+(b+d)g-(b+d)$.

As $t$ is an eigenvalue of valuation $d$, we have $\ov{M}_{d} \supset \ov{V}_{r}$.
Lemma \ref{subair1} gives us $\ov{M}_{b+2d} \supset U'$ and $U \subset \ov{M}_{2d}$.
By Theorem \ref{stouo}, we have $\dim(\ov{M}_{d})\geq f$, $\dim(\ov{M}_{2d}) \geq g+1$, and $\dim(\ov{M}_{b+2d})\geq g$.

Now,
$d(f-(g+1))+(2d)(g+1-g)+(b+2d)(g-1)=v_{\lp}(|K|)$.
So by Lemma \ref{main}, setting $j=3$, $s_{1}=f$, $s_{2}=g+1$, $s_{3}=g$, $s_{4}=\dim(\ov{\ker(L)})=1$, $t_{1}=d$, $t_{2}=2d$, and $t_{3}=b+2d$, we have $e_{0}=g+1$, $e_{d}=f-g-1$, $e_{2d}=1$, $e_{b+2d}=g-1$, and $e_{i}=0$ for all other $i$. 

\section*{Acknowledgements} This work was partially supported by a grant from the Simons Foundation (\#204181 to Peter Sin). We thank the anonymous referees for valuable suggestions and comments.
\bibliographystyle{plain}
\bibliography{Polargraph} 
\end{document}